\documentclass{amsart}

\usepackage{graphicx}
\usepackage[numbers]{natbib}
\usepackage{color,soul}
\usepackage[all]{xy}
\usepackage{subfigure}
\usepackage{geometry,mathtools}
\usepackage{amsthm}
\usepackage{float}

\newtheorem{theorem}{Theorem}[section]

\newtheorem{lemma}[theorem]{Lemma}

\newtheorem{definition}[theorem]{Definition}
\newtheorem{example}[theorem]{Example}

\newtheorem{prop}[theorem]{Proposition}
\newtheorem{corollary}[theorem]{Corollary}
\newtheorem{property}[theorem]{Property}
\theoremstyle{remark}
\newtheorem{remark}[theorem]{Remark}

\numberwithin{equation}{section}

\begin{document}

\title{Studies of distance one surgeries on the lens space $L(p,1)$}

\author{Zhongtao Wu \and Jingling Yang}
\address{Department of Mathematics, The Chinese University of Hong Kong, Hong Kong}
\email{ztwu@math.cuhk.edu.hk}
\email{yangjinglingm@gmail.com (Corresponding author)}
\thanks{} 






\begin{abstract}
In this paper, we study distance one surgeries between lens spaces $L(p,1)$ with $p \geq 5$ prime and lens spaces $L(n,1)$ for $n \in \mathbb{Z}$ and band surgeries from $T(2,p)$ to $T(2,n)$. In particular, we prove that $L(n,1)$ is obtained by a distance one surgery from $L(5,1)$ only if $n=\pm 1$, $4$, $\pm 5$, $6$ or $\pm 9$, and $L(n,1)$ is obtained by a distance one surgery from $L(7,1)$ if and only if $n=\pm 1$, $3$, $6$, $7$, $8$ or $11$.
\end{abstract}

\maketitle

\section{Introduction}


Dehn surgery is a fundamental operation in 3-manifold topology that enables one to modify the shape of 3-manifolds. An outstanding conjecture by Berge, which was recently (partially) solved by Greene \cite{Greene}, lists all the possible lens spaces that can be obtained by a Dehn surgery along a knot in the 3-sphere. Instead of the 3-sphere, a natural generalization of the above theorem is to list all the possible lens spaces that can be obtained by a Dehn surgery from other lens spaces.
The celebrated cyclic surgery theorem says that if the knot complement of a knot $K$ in a lens space is not a Seifert fiber space and $K$ admits another lens space surgery, then it must be a distance one surgery, that is, the surgery slope intersects the meridian of $K$ geometrically once. As Seifert fibered structures in lens spaces are well understood, we thus focus on distance one surgery between lens spaces.

In this paper, we are specifically concerned with distance one surgeries between the lens space $L(p,1)$ with $p \geq 5$ prime and lens spaces of type $L(n,1)$ for $n \in \mathbb{Z}$. This question is also motivated from DNA topology.  Recall that in biology, circular DNA can be modeled as a knot or link, and torus knots or links $T(2,n)$ are a family of DNA knot or link occurring frequently in biological experiments. Additionally, there exist enzymatic complexes that mediate DNA recombination, during which strands of DNA are exchanged and the topology of the DNA molecule may be altered in the process. To better understand the mechanism of DNA recombination, band surgery is used to model these enzymatic actions. Here {\it band surgery} on a knot or link $L$ is defined as follows: embed an unit square $I \times I$ into $S^3$ by $b:I \times I \rightarrow S^3$ such that $L \cap (I \times I)=b(\partial I \times I)$, then replace $L$ by $L'= (L-b(\partial I \times I)) \cup b(I \times \partial I)$. A fruitful technique of studying band surgery between knots or links is by lifting to their double branched covers. The double branched cover of $T(2,n)$ is the lens space $L(n,1)$.  As a consequence of the Montesinos trick, band surgeries on knots and links lift to distance one Dehn surgeries in their double branched covers. This explains the biological motivation to study distance one surgeries between lens spaces of type $L(n,1)$.  Finally, we remark that it is due to technical reasons that we only consider surgeries from $L(p,1)$ with prime number $p$: In such cases, any homologically essential knot $K$ in $L(p,1)$ is primitive, which makes $\underline{\rm Spin^c}(L(p,1),K)$ easier to study.  

Now, we list our main results.  The first theorem gives a complete answer when $n$ is even.
\begin{theorem}
\label{theoremeven}
The lens space $L(n,1)$ with $n$ even is obtained from a distance one surgery along a knot in $L(p,1)$ with $p \geq 5$ prime if and only if $n$ is $p+1$ or $p-1$.
\end{theorem}

The case for  an odd integer $n$ is more challenging.  Recall that $H_1(L(p,1))=\mathbb{Z}/p$.  Although every nonzero element is a generator of this cyclic group, there is a special element $[c]$ in $\mathbb{Z}/p$ that is given by the core of the either solid torus in the standard genus-$1$ Heegaard splitting of $L(p,1)$. 
Our theorem is divided into 3 parts according to the different homology classes that $K$ represents.

\begin{theorem}
\label{theoremgen}
Let $K$ be a knot in $Y=L(p,1)$ with $p \geq 5$ prime.
\begin{enumerate}
\item[(\romannumeral1)] Suppose $K$ is null-homologous.  The lens space $L(n,1)$ with $n$ odd is obtained by a distance one surgery along $K$ if and only if $n=p$, or $p=5$ and $n=-5$.
\item[(\romannumeral2)] Suppose $K$ is a homologically essential knot in $Y$ with $[K]=1 \cdot [c] \in H_1(Y)$.  The lens space $L(n,1)$ with $n$ odd is obtained by a distance one surgery along $K$ only if $n=\pm 1$ or $p=5$ and $n=-9$.
\item[(\romannumeral3)] Suppose $K$ is a homologically essential knot in $Y$ with $[K]=k \cdot [c] \in H_1(Y)$ and $k>1$.  If $L(n,1)$ with $n$ odd is obtained by a distance one surgery along $K$, then the slope is $(m \mu+\lambda)$ and $m < k + 3$.
\end{enumerate}
\end{theorem}

As the crossing numbers of DNA knots or links are often small, we also give some results about the lens spaces $L(5,1)$ and $L(7,1)$.

\begin{theorem}
\label{theorem57}
\quad\vspace{0.3mm}
\begin{enumerate}
\item[(\romannumeral1)] The lens space $L(n,1)$ is obtained by a distance one surgery from $L(5,1)$ only if $n=\pm 1$, $4$, $\pm 5$, $6$ or $\pm 9$.
\item[(\romannumeral2)] The lens space $L(n,1)$ is obtained by a distance one surgery from $L(7,1)$ if and only if $n=\pm 1$, $3$, $6$, $7$, $8$ or $11$.
\end{enumerate}
\end{theorem}

The above theorems about distance one surgeries plus the band surgeries we construct in Figure \ref{bandsurgery} readily imply the following corollaries about band surgeries once we lift to the double branched covers.

\begin{corollary}
\label{corbandsurgery}
\quad\vspace{0.3mm}
\begin{enumerate}
\item[(\romannumeral1)] The torus link $T(2,n)$ is obtained by a band surgery from $T(2,5)$ only if $n=\pm 1$, $4$, $\pm 5$, $6$ or $\pm 9$.
\item[(\romannumeral2)] The torus link $T(2,n)$ is obtained by a band surgery from $T(2,7)$ if and only if $n=\pm 1$, $3$, $6$, $7$, $8$ or $11$.
\item[(\romannumeral3)] The torus link $T(2,n)$ with $n$ even is obtained from $T(2,p)$ with $p \geq 5$ prime by a band surgery if and only if $n$ is $p+1$ or $p-1$.
\end{enumerate}
\end{corollary}

\begin{figure}[!h]
\centering
\subfigure[Band surgery from $T(2,p)$ to $T(2,p)$]{\label{bandsurgery1}
\includegraphics[width=0.375\textwidth]{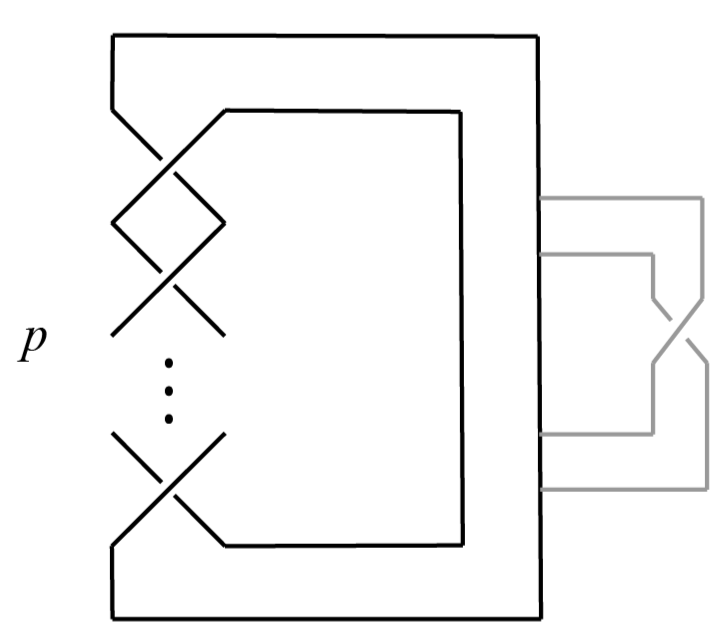}}\qquad
\subfigure[Band surgery from $T(2,p)$ to $T(2,p+4)$]{\label{bandsurgery5}
\includegraphics[width=0.38\textwidth]{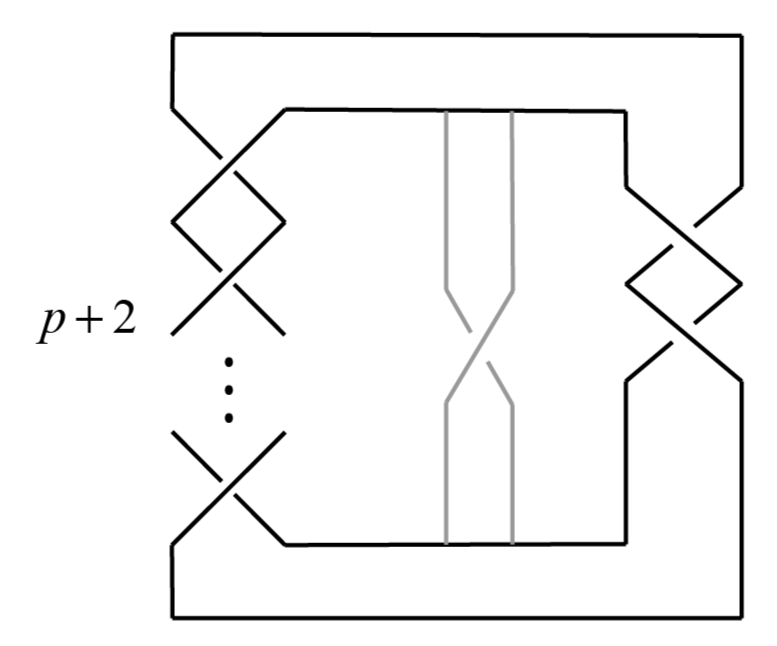}}\\
\subfigure[Band surgery from $T(2,p)$ to $T(2,p-1)$]{\label{bandsurgery2}
\includegraphics[width=0.26\textwidth]{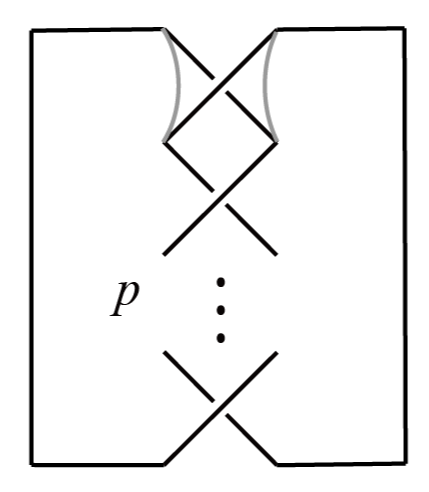}}\qquad
\subfigure[Band surgery from $T(2,p)$ to $T(2,p+1)$]{\label{bandsurgery3}
\includegraphics[width=0.22\textwidth]{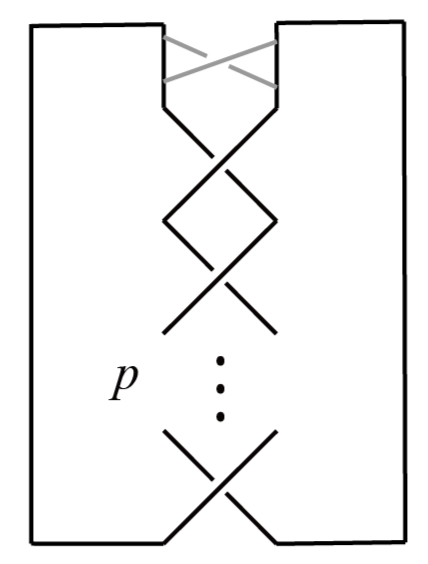}}\qquad
\subfigure[Band surgery from $T(2,p)$ to the unknot]{\label{bandsurgery4}
\includegraphics[width=0.255\textwidth]{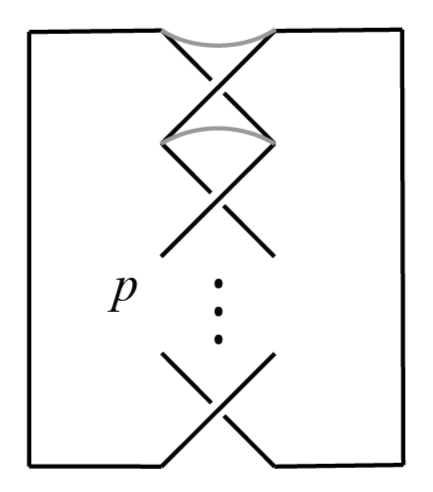}}
\caption{Examples of band surgeries, which lift to distance one surgery in double branched covers.}\label{bandsurgery}
\end{figure}

We now explain the connection and compare the methods of our paper with the existing ones in this direction.   In \cite{LMV}, Lidman, Moore and Vazquez classified distance one surgeries on $L(3,1)$ and the corresponding band surgeries on trefoil knot $T(2,3)$. A knot in $L(3,1)$ is either null-homologous or homologically essential. For null-homologous knots, they simply need to apply the $d$-invariant surgery formula essentially due to \cite{NiWu}.  For homologically essential knots, they have to work harder to first deduce a $d$-invariant surgery formula for $L(3,1)$ and then apply it to obstruct distance one surgeries between $L(3,1)$ and $L(n,1)$.  In our paper, we further generalize their $d$-invariant surgery formula for homologically essential knots in $L(3,1)$ to a knot in $L(p,1)$ with $p \geq 5$ prime. Then we use our new $d$-invariant surgery formula for homologically essential knots and the old formula for null-homologous knots to obstruct those pairs of lens spaces that are not arisen from the double branch cover of the knot pairs related by band surgeries exhibited in Figure \ref{bandsurgery}.


Like the proof of the $d$-invariant surgery formula for $S^3$, the key points to deduce the $d$-invariant surgery formula for homologically essential knots in $L(p,1)$ are: (1) Choose a special (relative) $\rm Spin^c$ structure so that we can find the element of minimal grading in the mapping cone. (2) Use a knot with simple knot Floer complex to fix the grading shift of the mapping cone. For (1), we choose the same relative $\rm Spin^c$ structure $\xi_0$ as in \cite{LMV}, which has a nice symmetric property shown in Lemma \ref{propsym}. With this property, we can easily trace the place where the minimal grading is supported.
For (2), there are multiple homology classes of knots in $L(p,1)$ instead of a single nontrivial class in $L(3,1)$ up to symmetry. For each of these homology classes, we use a so-called {\it simple knot} in $L(p,1)$ to fix the grading shift of the mapping cone.  In many cases, surgeries along simple knots in $L(p,1)$ produce a Seifert fiber space instead of a lens space, so we also need to deal with the computation of $d$-invariants of a Seifert fiber space, which is a substantial amount of extra work compared to \cite{LMV}.

This paper is structured as follows: Section \ref{Preliminaries} provides some preliminaries including homological analysis and basic properties of the $d$-invariant. Section \ref{Surgeries along null-homologous knots} introduces the $d$-invariant surgery formula for null-homologous knots and uses this formula to study distance one surgery along null-homologous knots in $L(p,1)$. In Section \ref{$d$-invariant surgery formula for homologically essential knots}, we deduce the $d$-invariant surgery formula for homologically essential knots in $L(p,1)$ from the mapping cone formula. In Section \ref{Surgeries along homologically essential knots}, we use our $d$-invariant surgery formula to study distance one surgeries along homologically essential knots in $L(p,1)$. Finally, we study distance one surgeries on the lens spaces $L(5,1)$ and $L(7,1)$ in Section \ref{Distance one surgeries on $L(5,1)$ and $L(7,1)$}.

The authors are partially supported by grant from the Research Grants Council of Hong Kong Special
Administrative Region, China (Project No. 14309016 and 14301317).

\section{Preliminaries}
\label{Preliminaries}
\subsection{Homological analysis}
\label{Homological analysis}
We adopt the convention that the lens space $L(p,q)$ is obtained from $p/q$-surgery of the unknot in $S^3$. Let $K$ be a knot in $Y=L(p,1)$ with $p \geq 5$ prime. Then the homology class of $K$ in $Y$ is either trivial or a generator of $H_1(Y)$. In the former case $K$ is called \emph{null-homologous}, and in the latter case $K$ is called \emph{homologically essential}.

When $K$ is null-homologous, there is a canonical way to fix the meridian and the longitude of $K$. Denote by $Y_{m}(K)$ the manifold obtained from $m$-surgery along $K$. We have $H_1(Y_{m}(K))=\mathbb{Z}/p \oplus \mathbb{Z}/m$.


When $K$ is homologically essential, there are $\frac{p-1}{2}$ different homology classes of $K$ in $H_1(Y)=\mathbb{Z}/p$ up to symmetry. We represent $Y$ by a Kirby diagram with an unknot $U$ with framing $p$ in $S^3$; hence $Y=V_0\cup V_1$ where $V_0$ denotes the solid torus that is the complement of $U$, and $V_1$ denotes the solid torus that is glued on. Fix an orientation of the unknot $U$. Let $c$ be the core of $V_0$, and we fix an orientation of $c$ such that the linking number of $c$ and $U$ equals 1. Then we can choose an orientation of $K$ such that $[K]=k[c] \in H_1(L(p,1))$ for some integer $ 1 \leq k \leq \frac{p-1}{2}$.
We call this $k$ the \emph{mod p-winding number}, or simply the \emph{winding number} of $K$. By possibly handlesliding $K$ over $U$ in the Kirby diagram, which is equivalent to isotopying $K$ in $Y$ over the meridian of $V_1$, we may further assume that the linking number of $K$ and $U$ is exactly $k$.  We fix our meridian $\mu$ and longitude $\lambda$ for $K$ by regarding $K$ as a component of the link consisting of $K$ and $U$ in $S^3$. Also we fix the meridian $\mu_0$ and longitude $\lambda_0$ for the unknot $U$. Then the first homology
$$H_1(S^3-U-K)=\mathbb{Z} \langle \mu_0 \rangle \oplus \mathbb{Z} \langle \mu \rangle \quad {\rm and} \quad \left \{\begin{matrix}
[\lambda_0]=k \cdot [\mu]\\
[\lambda]=k \cdot [\mu_0]
\end{matrix}\right. .$$
Therefore, the first homology of the knot complement $Y-K$ is
$$H_1(Y-K)=H_1(S^3-U-K)/ \langle p\mu_0 + \lambda_0 \rangle= H_1(S^3-U-K)/ \langle p\mu_0 + k \mu \rangle .$$
Let $\theta= p' \mu_0+k'\mu$, where $pk'-kp'=1$, then $H_1(Y-K)=\mathbb{Z} \langle \theta \rangle$. One may check that
\begin{equation}
[\mu]=p[\theta] \in H_1(Y-K),
\end{equation}
\begin{equation}
[\lambda]=-k^2[\theta] \in H_1(Y-K).
\end{equation}
Since we are interested in distance one surgery, we only consider $(m \cdot \mu + \lambda)$-surgery. Denote by $Y_{m \cdot \mu+ \lambda}(K)$ the surgered manifold of $(m \cdot \mu + \lambda)$-surgery along $K$. Then the first homology
\begin{equation} \label{firsthomology}
H_1(Y_{m \cdot \mu+ \lambda}(K))=\mathbb{Z}/|pm-k^2|.
\end{equation}
For simplicity, we will also refer the above surgery as the $m$-surgery along $K$ with the understanding that $\mu$ and $\lambda$ are chosen as just described.

\medskip
The lemma below will be repeatedly used in the later sections.

\begin{lemma}
\label{lemma1}
Let $Y'$ be the manifold obtained by a distance one surgery from $Y=L(p,1)$ with $p \geq 5$ prime, and let $W: Y \rightarrow Y'$ be the associated cobordism. Then $|H_1(Y')|$ is even if and only if $W$ is Spin.
\end{lemma}

\begin{figure}[!h]
\begin{minipage}{0.46\linewidth}
 \centerline{\includegraphics[width=4.3cm]{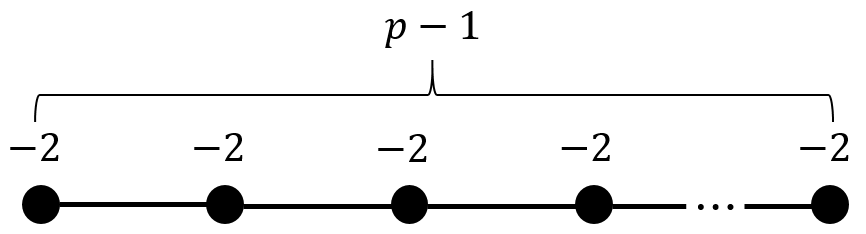}}
\end{minipage}
\vspace{7mm}
\begin{minipage}{.43\linewidth}
$Q_{X \cup W}=
\begin{bmatrix}
 -2 & 1 &    &       &        &         & a_1\\
 1  & -2& 1  &       &        &         & a_2\\
    & 1 & -2 &       &        &         & a_3\\
    &   &    & \ddots&        &         &\vdots\\
    &   &    &       &  -2    & 1       & a_{p-2}\\
    &   &    &       &    1   & -2      & a_{p-1}\\
 a_1&a_2& a_3& \dots & a_{p-2}& a_{p-1} & m
\end{bmatrix}_{p \times p}$
\end{minipage}
\caption{The plumbing diagram of $X$ and the intersection form of $X\cup W$}
\label{spinproof}
\end{figure}

\begin{proof}
We will make use of the fact that a 4-manifold whose first homology has no 2-torsion is Spin if and only if its intersection form is even. Consider a plumbed 4-manifold $X$ with the plumbing diagram depicted in Figure \ref{spinproof}, and $\partial X=L(p,1)$.  As $X$ is simply-connected and has even intersection form, it is a Spin 4-manifold. Attach the cobordism $W$ to $X$ along $L(p,1)$. Then the simply-connected 4-manifold $X \cup W$ is Spin if and only if $W$ is Spin, since $L(p,1)$ is a $\mathbb{Z}/2$ homology sphere, $H_1(W)$ has no 2-torsion and $X$ is Spin.
We thus compute the intersection form of $X \cup W$, $Q_{X\cup W}$ in Figure \ref{spinproof}, where $m$ represents the surgery coefficient on the knot $K$ and $a_i$'s are some integers.

%

We claim that ${\rm det}(Q_{X \cup W})$ is even if and only if $m$ is even.  We show this by expanding along the last row and then along the last column of the matrix $Q_{X \cup W}$.  To this effect, let $Q_i$ denote the $(p-1)\times (p-1)$ matrix obtained by removing the $i^{th}$ column and the last row from $Q_{X \cup W}$, and $Q_{i,j}$ denote the $(p-2)\times (p-2)$ matrix obtained by removing the $j^{th}$ row and the last column from $Q_i$.  Then,

\begin{align*}
\det (Q_{X \cup W})&\equiv \sum_{i=1}^{p-1} (a_i \cdot  \det Q_i) + m \cdot \det Q_{p} \\
&\equiv \sum_{i=1}^{p-1} ( a_i \cdot  (\sum_{j=1}^{p-1} a_j\cdot \det Q_{i,j}) ) + m \cdot  \det Q_{p} \\
&\equiv  \sum^{p-1}_{i=1}(a_i^2 \cdot \det Q_{i,i}) + m \cdot \det Q_{p}  \\
&\equiv m \pmod{2}, \\
\end{align*}
where the third equality follows from the symmetric property $\det Q_{i,j}=\det Q_{j,i}$ for all $i,j$, and the last equality follows from the fact that $\det Q_{p}=p$ is odd and $\det Q_{i,i}$ is even for $1\leq i\leq p-1$.  



Hence, we see that $|H_1(Y')|={\rm det}(Q_{X \cup W})$ is even if and only if $m$ is even if and only if $Q_{X \cup W}$ is even if and only if $X \cup W$ is $\rm Spin$. Therefore $|H_1(Y')|$ is even if and only if $W$ is Spin.
\end{proof}

\subsection{$d$-invariant}

For a rational homology sphere $Y$ equipped with a $\rm Spin^c$ structure $\mathfrak{t}$, the {\it$d$-invariant}, or {\it correction term}, denoted by $d(Y, \mathfrak{t})$, is the minimal $\mathbb{Q}$-grading of the image of $HF^{\infty}(Y, \mathfrak{t})$ in $HF^+(Y, \mathfrak{t})$. We refer the reader to Ozsv\'{a}th-Szab\'{o} \cite{OS1} for details and cite the following recursive formula for the $d$-invariant of a lens space.

\begin{theorem}
\label{thmOS1}
Let $p>q>0$ be relatively prime integers. Then there exists an identification ${\rm Spin^c}(L(p,q)) \cong \mathbb{Z}/p$ such that
\begin{equation}
\label{eq1}
d(L(p,q),i)=-\dfrac{1}{4}+\dfrac{(2i+1-p-q)^2}{4pq}-d(L(q,r),j)
\end{equation}
where $r$ and $j$ are the reductions of $p$ and $i$ (mod $q$) respectively.
\end{theorem}

Under the identification ${\rm Spin^c}(L(p,q)) \cong \mathbb{Z}/p$ in Theorem \ref{thmOS1}, the self-conjugate $\rm Spin^c$ structures on $L(p,q)$ correspond to the integers amongst $\dfrac{p+q-1}{2}$ and $\dfrac{q-1}{2}$. For a lens space $L(n,1)$ with $n>0$, if $n$ is odd, then there exists only one self-conjugate $\rm Spin^c$ structure corresponding to $0$; if $n$ is even, then there are two self-conjugate $\rm Spin^c$ structures corresponding to $\dfrac{n}{2}$ and $0$.

By the recursive formula above, we can compute the values of $d(L(n,1),i)$ for $n>0$ as follows.
\begin{equation}
d(L(n,1),i)=-\dfrac{1}{4}+\dfrac{(2i-n)^2}{4n}   \label{eq2}
\end{equation}

\medskip
When the cobordism associated to a distance one surgery between two rational homology spheres is Spin, there is a strong constraint on the $d$-invariant.

\begin{lemma}[Lidman-Moore-Vazquez, Lemma 2.7 in \cite{LMV}]
\label{lemmaLMV}
Let $(W,\mathfrak{s}):(Y,\mathfrak{t}) \rightarrow (Y',\mathfrak{t}')$ be a Spin cobordism between L-spaces satisfying $b_2^+(W)=1$ and $b_2^-(W)=0$. Then
\[d(Y',\mathfrak{t}')-d(Y,\mathfrak{t})=-\dfrac{1}{4}.\]
\end{lemma}

Now we are ready to show Theorem \ref{theoremeven}. The proof runs a very similar argument to Lidman-Moore-Vazquez \cite[Proposition 2.6]{LMV}.   


\begin{proof}[Proof of Theorem \ref{theoremeven}]
First we prove the ``only if'' part of the theorem. By homological obstruction, distance one surgery on $L(p,1)$ can not give $S^1 \times S^2$, so we assume that $n \neq 0$.  Let $W$ be the associated 2-handle cobordism between $L(p,1)$ and $L(n,1)$.  Then $H_2(W)=\mathbb{Z}$.  We claim that $W$ is either positive definite (i.e. $b^+_2(W)=1$ and $b^-_2(W)=0$) or negative definite (i.e. $b^+_2(W)=0$ and $b^-_2(W)=1$), that is, the case $b^+_2(W)=b^-_2(W)=0$ can never happen. Let $N$ be a 4-manifold with boundary $L(p,1)$, which is obtained by attaching a $p$-framed 2-handle to $B^4$ along an unknot. Let $Z$ denote the 4-manifold obtained by attaching $N$ to $W$. Then $b^{\pm}_2(Z)=b^{\pm}_2(N)+b^{\pm}_2(W)$. We see that $W$ satisfies $b^+_2(W)=b^-_2(W)=0$ if and only if $b^+_2(Z)=1$ and $b^-_2(Z)=0$. Note that the intersection form $Q_Z$ of $Z$ is
$$Q_{Z}=
\begin{bmatrix}
p & k\\
k & m
\end{bmatrix}.$$
As $pm-k^2 = n \neq 0$, we can never have $b^+_2(W)=b^-_2(W)=0$.

Since $|H_1(L(n,1))|=|n|$ is assumed to be even, Lemma \ref{lemma1} implies that the associated cobordism $W$ is Spin. If $b^+_2(W)=1$ and $b^-_2(W)=0$, then by Lemma \ref{lemmaLMV} we have
\begin{equation}
d(L(n,1), i)-d(L(p,1),0)=-\dfrac{1}{4}, \label{eq5}
\end{equation}
where $i=0$ or $\frac{|n|}{2}$. Applying Equation (\ref{eq2}) to $L(p,1)$, we conclude that $d(L(n,1), i)=\dfrac{p}{4}-\dfrac{1}{2}$ for $i=0$ or $\frac{|n|}{2}$. If $i=0$, Equation (\ref{eq2}) gives

\[d(L(n,1),0)= \left\{\begin{matrix}
\dfrac{n-1}{4}  \qquad  n>0\\
\dfrac{1+n}{4}   \qquad n<0
\end{matrix}\right. .\]

Therefore, $d(L(n,1), 0)=\dfrac{p}{4}-\dfrac{1}{2}$ may only occur when $n=p-1$. If $i=\dfrac{|n|}{2}$, Equation (\ref{eq2}) gives
\[d(L(n,1),\dfrac{|n|}{2})= \left\{\begin{matrix}
-\dfrac{1}{4} \qquad  n>0\\
\dfrac{1}{4}  \qquad \,\,  n<0
\end{matrix}\right. .\]
which can never equal $\dfrac{p}{4}-\dfrac{1}{2}$.

If $b^+_2(W)=0$ and $b^-_2(W)=1$, applying Lemma \ref{lemmaLMV} to $-W$ we obtain
\begin{equation}
d(L(p,1),0)-d(L(n,1), i)=-\dfrac{1}{4}, \label{eq6}
\end{equation}
where $i=0$ or $\frac{|n|}{2}$. By a similar calculation, we find that (\ref{eq6}) holds only when $n=p+1$. So we complete the proof of the ``only if'' part.

The ``if'' part is true because there exist band surgeries from $T(2,p)$ to $T(2,p+1)$ and band surgeries from $T(2,p)$ to $T(2,p-1)$ shown in Figure \ref{bandsurgery2} and \ref{bandsurgery3}.  The double branched cover of those band surgeries gives the desired distance one surgery.  This completes the proof.
\end{proof}



\section{Surgeries along null-homologous knots}
\label{Surgeries along null-homologous knots}
\subsection{$d$-invariant surgery formula for null-homologous knots}

For any null-homologous knot $K$ in a rational homology sphere $Y$, there exists a non-negative integer $V_{\mathfrak{t},i}$ associated to $K$ for each $i \in \mathbb{Z}$ and $\mathfrak{t} \in {\rm Spin^c}(Y)$ satisfying the following property.

\begin{property}[Proposition 7.6 in \cite{Ras1}]
\label{monotonicity}
\[V_{\mathfrak{t},i} \geq V_{\mathfrak{t},i+1} \geq V_{\mathfrak{t},i}-1.\]
\end{property}

Let $\mathfrak{t}_i\in {\rm Spin^c}(Y_m(K)) $ denote the ${\rm Spin^c}$ structure that corresponds to $(\mathfrak{t},i) \in {\rm Spin^c}(Y)\oplus \mathbb{Z}_m$ under the natural bijection between  the sets of $\rm Spin^c$ structures ${\rm Spin^c}(Y_m(K))$ and ${\rm Spin^c}(Y) \oplus \mathbb{Z}_m$.
Ni and the first author give a $d$-invariant surgery formula for a knot in $S^3$ \cite[Proposition 1.6]{NiWu}, whose argument also applies to a general null-homologous knot in an $L$-space.





\begin{prop}
\label{NiWu}
Fix an integer $m>0$ and a self-conjugate $\rm Spin^c$ structure $\mathfrak{t}$ on an L-space $Y$. Let $K$ be a null-homologous knot in $Y$. Then, for any $\mathfrak{t}_i \in {\rm Spin^c}(Y_m(K))$,
\begin{equation}
d(Y_m(K), \mathfrak{t}_i)=d(Y,\mathfrak{t})+d(L(m,1),i)-2N_{\mathfrak{t},i}, \label{eq4}
\end{equation}
where $N_{\mathfrak{t},i}=\max\{V_{\mathfrak{t},i},V_{\mathfrak{t},m-i}\}$.
\end{prop}


Lidman, Moore and Vazquez also give the following lemma which will be used repeatedly when we apply the above $d$-invariant surgery formula.

\begin{lemma}[Lidman-Moore-Vazquez \cite{LMV}]
\label{t0LMV}
Let $K$ be a null-homologous knot in a $\mathbb{Z}/2$ homology sphere $Y$ and $\mathfrak{t}$ a self-conjugate $\rm Spin^c$ structure on $Y$. Let $\mathfrak{t}_0$ be the $\rm Spin^c$ structure on $Y_m(K)$ as described in Proposition \ref{NiWu}. Then $\mathfrak{t}_0$ is self-conjugate.
\end{lemma}

\subsection{Surgeries along null-homologous knots}
\label{sub Surgeries along null-homologous knots}
In this section, we study surgeries along null-homologous knots in $L(p,1)$ with $p \geq 5$ prime. To prove Theorem \ref{theoremgen} (i), we may assume that the surgered manifold is the lens space $L(n,1)$ with $n=\pm pm$ for some odd integer $m$.   The argument is adapted from  \cite[Section 3.1]{LMV}, where an analogous statement for surgeries along null-homologous knots in $L(3,1)$ is proved.




\begin{prop}
\label{propofnull1}
If $m \geq 3$ is odd, then $L(pm,1)$ cannot be obtained by $m$-surgery along a null-homologous knot in $L(p,1)$ with $p \geq 5$ prime.
\end{prop}

\begin{proof}
Suppose $L(pm,1)$ is obtained by $m$-surgery along a null-homologous knot in $L(p,1)$, and $m \geq 3$ is odd.
Since $p$ is odd, the $\rm Spin^c$ structure corresponding to $0$ is the unique self-conjugate one on $L(p,1)$. Choose the self-conjugate $\rm Spin^c$ structure $\mathfrak{t}=0$ on $L(p,1)$ and let $i=0$ in Formula (\ref{eq4}). Then we have
\begin{equation}
\label{eq7}
d(L(pm,1),\mathfrak{t}_0)=d(L(p,1),0)+d(L(m,1),0)-2N_{0,0},
\end{equation}
where the first 0 in the subscript of $N_{0,0}$ stands for the self-conjugate $\rm Spin^c$ structure corresponding to $0$ on $L(p,1)$ and the second 0 in the subscript represents $i=0$.

Lemma \ref{t0LMV} implies that $\mathfrak{t}_0$ is a self-conjugate $\rm Spin^c$ structure on $L(pm,1)$.  Since $m$ and $p$ are odd integers, $\mathfrak{t}_0$ must be the unique self-conjugate $\rm Spin^c$ structure on $L(pm,1)$ that corresponds to 0 in the above identification with $\mathbb{Z}/pm$. Equation (\ref{eq7}) implies
\begin{align*}
N_{0,0}&=\dfrac{d(L(p,1),0)+d(L(m,1),0)-d(L(pm,1),0)}{2}\\
       &=\dfrac{(1-p)(m-1)}{8} < 0,
\end{align*}
which contradicts the fact that $N_{0,0}$ is non-negative.
\end{proof}

\begin{prop}
\label{propofnull2}
If $m \geq 1$ is odd, then $L(-pm,1)$ cannot be obtained by $-m$-surgery along a null-homologous knot in $L(p,1)$ with $p \geq 5$ prime.
\end{prop}

\begin{proof}
Suppose $L(-pm,1)$ is obtained by $-m$-surgery along a null-homologous knot in $L(p,1)$ with $m \geq 1$ odd. By reversing the orientation, $L(pm,1)$ is obtained by $m$-surgery on $L(-p,1)$ along a null-homologous knot. Choosing the self-conjugate $\rm Spin^c$ structure $\mathfrak{t}=0$ and $i=0$ in Formula (\ref{eq4}), we have
\begin{equation*}
d(L(pm,1),\mathfrak{t}_0)=d(L(-p,1),0)+d(L(m,1),0)-2N_{0,0}.
\end{equation*}
Also, $\mathfrak{t}_0$ is a self-conjugate $\rm Spin^c$ structure on $L(pm,1)$ that corresponds to $0$ in $\mathbb{Z}/pm$. Hence,
\begin{align*}
N_{0,0}&=\dfrac{d(L(-p,1),0)+d(L(m,1),0)-d(L(pm,1),0)}{2}\\
       &=\dfrac{(1-p)(m+1)}{8} < 0,
\end{align*}
which is a contradiction.
\end{proof}

\begin{prop}
\label{propofnull3}
If $m \geq 3$ is odd, then $L(pm,1)$ cannot be obtained by $-m$-surgery along a null-homologous knot in $L(p,1)$ with $p \geq 5$ prime.
\end{prop}

\begin{proof}
Suppose that $L(pm,1)$ is obtained by $-m$-surgery along a null-homologous knot in $L(p,1)$ with $m \geq 3$ odd. Reversing the orientation, we have $L(-pm,1)$ is given by $m$-surgery on $L(-p,1)$ along a null-homologous knot. Consider the self-conjugate $\rm Spin^c$ structure $\mathfrak{t}=0$. Formula (\ref{eq4}) gives us the following equation
\begin{equation*}
d(L(-pm,1),\mathfrak{t}_0)=d(L(-p,1),0)+d(L(m,1),0)-2N_{0,0},
\end{equation*}
where we choose $i=0$. Also by the same reason as above, the self-conjugate $\rm Spin^c$ structure $\mathfrak{t}_0$ corresponds to $0$ on $L(-pm,1)$, so
\begin{align*}
N_{0,0}&=\dfrac{d(L(-p,1),0)+d(L(m,1),0)+d(L(pm,1),0)}{2}\\
       &=\dfrac{(p+1)(m-1)}{8} >0,
\end{align*}
where $N_{0,0}=\max\{V_{0,0}, V_{0,m}\}=V_{0,0}$ by the monotonicity of $V_{\mathfrak{t},i}$. Now we choose $\mathfrak{t}=0$ and $i=1$ in Formula (\ref{eq4}). Then
\begin{equation}
\label{eq8}
d(L(-pm,1),\mathfrak{t}_1)=d(L(-p,1),0)+d(L(m,1),1)-2N_{0,1},
\end{equation}
where $N_{0,1}=\max\{V_{0,1},V_{0,m-1}\}=V_{0,1}$. By Property \ref{monotonicity},
$V_{0,1}=V_{0,0} \,\,\, {\rm or} \,\,\, V_{0,0}-1$.  

\medskip\noindent
Case \romannumeral1: $V_{0,1}=V_{0,0}=\dfrac{(p+1)(m-1)}{8}$.\\
Equation (\ref{eq2}) implies
$d(L(-pm,1), \mathfrak{t}_1)=-\left(-\dfrac{1}{4}+ \dfrac{(2j-pm)^2}{4pm}\right)$
 for some integer $j\in [0, pm-1]$.
Putting it into (\ref{eq8}), we have
\[-\dfrac{1}{4}+ \dfrac{(2j-pm)^2}{4pm}+\dfrac{1-p}{4}+\left(-\dfrac{1}{4}+\dfrac{(2-m)^2}{4m}\right)-\dfrac{(p+1)(m-1)}{4}=0\]
for some integer $j \in [0,pm-1]$, which can be further simplified to
\begin{equation}
\label{eq24}
j^2-pmj+p-mp=0.
\end{equation}
We claim that the function $f(x)=x^2-pmx+p-mp$ has no root in $[0,pm-1]$. Indeed, $f(x)<0$ for any $x \in [0,pm-1]$, since its axis of symmetry is $x=\frac{pm}{2}$ and $f(0)=p-mp<0$. So this case is impossible.

\medskip\noindent
Case \romannumeral2: $V_{0,1}=V_{0,0}-1=\dfrac{(p+1)(m-1)}{8}-1$. \\
Similarly, by (\ref{eq8}), there exists an integer $j \in [0,pm-1]$ satisfying

\[-\dfrac{1}{4}+ \dfrac{(2j-pm)^2}{4pm}+\dfrac{1-p}{4}+\left(-\dfrac{1}{4}+\dfrac{(2-m)^2}{4m}\right)-\dfrac{(p+1)(m-1)}{4}+2=0,\]
which can be simplified to
\begin{equation}
\label{eq25}
j^2-pmj+pm+p=0.
\end{equation}
Let $f(x)=x^2-pmx+pm+p$. Its axis of symmetry is $x=\frac{pm}{2}$. We see that $f(0)=pm+p>0$, $f(1)=p+1>0$ and $f(2)=p+4-pm<0$ since $m \geq 3$ and $p \geq 5$. Hence the roots of $f(x)$ are in the intervals $(1,2)$ and $(pm-2,pm-1)$, and they are not integers. Therefore, there is no integral root of $f(x)$ in $[0,pm-1]$, which gives a contradiction.

%
%
\end{proof}

\begin{prop}
\label{propofnull4}
If $m \geq 3$ is odd, then $L(-pm,1)$ cannot be obtained by $m$-surgery along a null-homologous knot in $L(p,1)$ with $p \geq 5$ prime.
\end{prop}

\begin{proof}
Suppose that $L(-pm,1)$ is obtained by $m$-surgery along a null-homologous knot in $L(p,1)$ with $m \geq 3$ odd. Applying Formula (\ref{eq4}) to the case that $\mathfrak{t}=0$ is self-conjugate on $L(p,1)$ and $i=0$, we have
\begin{equation*}
d(L(-pm,1),\mathfrak{t}_0)=d(L(p,1),0)+d(L(m,1),0)-2N_{0,0},
\end{equation*}
where $N_{0,0}=\max\{V_{0,0}, V_{0,m}\}= V_{0,0}$ by monotonicity of $V_{\mathfrak{t},i}$. Also, $\mathfrak{t}_0$ is the unique self-conjugate $\rm Spin^c$ structure corresponding to $0$ on $L(-pm,1)$, so the above equation implies

\begin{align*}
N_{0,0}=V_{0,0}&=\dfrac{d(L(p,1),0)+d(L(m,1),0)+d(L(pm,1),0)}{2}\\
               &=\dfrac{p+m+pm-3}{8} > 0.
\end{align*}

Next we choose the self-conjugate $\mathfrak{t}=0$ on $L(p,1)$ and $i=1$. Then by Formula (\ref{eq4})
\begin{equation}
\label{eq9}
d(L(-pm,1),\mathfrak{t}_1)=d(L(p,1),0)+d(L(m,1),1)-2N_{0,1},
\end{equation}

where $N_{0,1}=\max\{V_{0,1}, V_{0,m-1}\}=V_{0,1}$. By Property \ref{monotonicity},
$V_{0,1}=V_{0,0} \,\,\, {\rm or} \,\,\, V_{0,0}-1$. 

\medskip\noindent
Case \romannumeral1: $V_{0,1}=V_{0,0}=\dfrac{p+m+pm-3}{8}$.\\
Equation (\ref{eq2}) implies
$d(L(-pm,1), \mathfrak{t}_1)=-\left(-\dfrac{1}{4}+ \dfrac{(2j-pm)^2}{4pm}\right)$ for some integer $j\in [0, pm-1]$.
Plugging it into (\ref{eq9}), we have
\[-\dfrac{1}{4}+ \dfrac{(2j-pm)^2}{4pm}+\dfrac{p-1}{4}+\left(-\dfrac{1}{4}+\dfrac{(2-m)^2}{4m}\right)-\dfrac{p+m+pm-3}{4}=0.\]
The equation can be simplified to
\[j^2-pmj+p-mp=0,\]
which is the same as (\ref{eq24}). As there is no integer $j$ in $[0,pm-1]$ that satisfies the equation, we can rule out this case.

\medskip\noindent
Case \romannumeral2: $V_{0,1}=V_{0,0}-1=\dfrac{p+m+pm-3}{8}-1$. \\
Similarly, by (\ref{eq9}), there exists an integer $j \in [0,pm-1]$ satisfying
\[-\dfrac{1}{4}+ \dfrac{(2j-pm)^2}{4pm}+\dfrac{p-1}{4}+\left(-\dfrac{1}{4}+\dfrac{(2-m)^2}{4m}\right)-\dfrac{p+m+pm-3}{4}+2=0,\]
which can be simplified to
\[j^2-pmj+pm+p=0.\]
This is the same equation as (\ref{eq25}).  As there is no integer $j$ in $[0,pm-1]$ that satisfies the equation, we can rule out this case.
\end{proof}

The last ingredient for proving  Theorem \ref{theoremgen} (\romannumeral1) is the following result due to Moore and Vazquez.

\begin{prop}[Corollary 3.7 in \cite{MV}]
\label{propofnull5}
Suppose $m>0$ is a square-free odd integer. There exists a distance one surgery along any knot $K$ in $L(m,1)$ yielding $-L(m,1)$ if and only if $m=1$ or $m=5$.
\end{prop}

\begin{proof}[Proof of Theorem \ref{theoremgen} (\romannumeral1)]
By homological reasons, we may assume $n=\pm pm$ for some odd integer $m>0$.  Then, Propositions \ref{propofnull1} - \ref{propofnull4} imply that $n$ can only be $\pm p$. Lifting the band surgery shown in Figure \ref{bandsurgery1} to the double branched cover, we can see a distance one surgery along a null-homologous knot in $L(p,1)$ produces itself.  On the other hand, Proposition \ref{propofnull5} shows that a distance one surgery from $L(p,1)$ to $L(-p,1)$ exists if and only if $p=5$. This completes the proof.
\end{proof}

\section{$d$-invariant surgery formula for homologically essential knots}
\label{$d$-invariant surgery formula for homologically essential knots}
\subsection{The mapping cone for rationally null-homologous knots}


In this section, we give a $d$-invariant surgery formula for rationally null-homologous knots based on the mapping cone formula by Ozsv\'ath and Szab\'o \cite{OSr}.  We assume the readers are familiar with Heegaard Floer homology and we use $\mathbb{F}=\mathbb{Z}/2\mathbb{Z}$ coefficients throughout unless otherwise stated.

Let $Y$ be a rational homology sphere and $K$ an oriented knot in $Y$. There is a canonical choice of meridian $\mu$ of $K$, and a framing $\gamma$ is an embedded curve on the boundary of the tubular neighborhood of $K$ which intersects $\mu$ once transversely. We write $\underline{\rm Spin^c}(Y,K)$ for the relative $\rm Spin^c$ structures on $Y-K$, which has an affine identification with $H^2(Y,K)$. In particular, if the knot $K$ is primitive, i.e., $K$ generates $H_1(Y)$, then $\underline{\rm Spin^c}(Y,K)$ is affinely isomorphic to $\mathbb{Z}$.


Let $w$ be a vector field on $S^1 \times D^2$ as described in \cite{OSr}, which is also the so-called distinguished Euler structure in Turaev's literature \cite{Tur}. Gluing this vector field $w$ to a relative $\rm Spin^c$ structure on $Y-K$ gives us a natural map:
\[G_{Y, \pm K}: \underline{\rm Spin^c}(Y,K) \rightarrow {\rm Spin^c}(Y) \cong H^2(Y),\]
satisfying
\[G_{Y,\pm K}(\xi+\kappa)=G_{Y, \pm K}(\xi)+i^{\ast}(\kappa),\]
where $\kappa \in H^2(Y,K)$ and $i^{\ast}: H^2(Y,K) \rightarrow H^2(Y)$ is induced from inclusion. Here, $-K$ denotes $K$ with the opposite orientation. We have
\[G_{Y,-K}(\xi)=G_{Y,K}(\xi)+PD[K].\]

For each $\xi \in \underline{\rm Spin^c}(Y,K)$, there is a $\mathbb{Z} \oplus \mathbb{Z}$-filtered knot Floer complex $C_{\xi}=CFK^{\infty}(Y,K,\xi)$, whose bifiltration is given by $(i,j)=(algebraic, Alexander)$. Let $A^+_{\xi}=C_{\xi}\{\max\{i,j\} \geq 0\}$ and $B^+_{\xi}=C_{\xi}\{i \geq 0\}$. There are two natural projection maps
\[v^+_{\xi}: A^+_{\xi} \rightarrow B^+_{\xi}, \quad h^+_{\xi}: A^+_{\xi} \rightarrow B^+_{\xi+PD[\gamma]}.\]
Ozsv\'{a}th and Szab\'{o} show that $v^+_{\xi}$ and $h^+_{\xi}$ correspond to the negative definite cobordism maps $W'_n: Y_{\gamma+n\mu}(Y) \rightarrow Y$ for $n \gg 0$ equipped with certain $\rm Spin^c$ structures. See \cite[Theorem 4.1]{OSr} for details.

The Heegaard Floer homology of any $\rm Spin^c$ rational homology sphere contains a non-torsion submodule $\mathcal{T}^+=\mathbb{F}[U,U^{-1}]/U \cdot \mathbb{F}[U]$, called the {\it tower}. On the level of homology, both $v^+_{\xi}$ and $h^+_{\xi}$ induce grading homogeneous maps between towers, which are multiplication by $U^{N}$ for some integer $N \geq 0$. We denote the corresponding non-negative integers for $v^+_{\xi}$ and $h^+_{\xi}$ by $V_{\xi}$ and $H_{\xi}$ respectively, which are also known as the local $h$-invariants of Rasmussen \cite{Ras1}. An analogue of Property \ref{monotonicity} shows that for each $\xi \in \underline{\rm Spin^c}(Y,K)$,
\begin{equation}
\label{eq11}
V_{\xi} \geq V_{\xi+PD[\mu]} \geq V_{\xi}-1.
\end{equation}

Given any $\mathfrak{s} \in {\rm Spin^c}(Y_{\gamma}(K))$, let
\begin{align*}
\mathbb{A^+_{\mathfrak{s}}}&=\mathop{\bigoplus}_{\{\xi \in \underline{\rm Spin^c}(Y_{\gamma}(K),K_{\gamma})| G_{Y_{\gamma}(K),K_{\gamma}}(\xi)=\mathfrak{s}\}} A^+_{\xi}\\
\mathbb{B^+_{\mathfrak{s}}}&=\mathop{\bigoplus}_{\{\xi \in \underline{\rm Spin^c}(Y_{\gamma}(K),K_{\gamma})| G_{Y_{\gamma}(K),K_{\gamma}}(\xi)=\mathfrak{s}\}} B^+_{\xi},
\end{align*}
where $K_{\gamma}$ denotes the oriented dual knot of the knot $K$ in the surgered manifold $Y_{\gamma}(K)$, and $G_{Y_{\gamma}(K),K_{\gamma}}: \underline{\rm Spin^c}(Y_{\gamma}(K),K_{\gamma}) \rightarrow {\rm Spin^c}(Y_{\gamma}(K))$. Note that $\underline{\rm Spin^c}(Y,K)=\underline{\rm Spin^c}(Y_{\gamma}(K),K_{\gamma})$, since they both represent the set of the relative ${\rm Spin^c}$ structures on the knot complement $Y-K=Y_{\gamma}(K)-K_{\gamma}$. Let

\[D^+_{\mathfrak{s}}: \mathbb{A}^+_{\mathfrak{s}} \rightarrow \mathbb{B}^+_{\mathfrak{s}}, \quad
(\xi, a) \mapsto (\xi, v^+_{\xi}(a))+(\xi+PD[\gamma], h^+_{\xi}(a))\]
The knot Floer complex of the knot $K$ and the Heegaard Floer homology of the manifold obtained from distance one surgery along $K$ are related by:

\begin{theorem}[Ozsv\'{a}th-Szab\'{o}, Theorem 6.1 in \cite{OSr}]
\label{mappingconethm}
For any $\mathfrak{s} \in {\rm Spin^c}(Y_{\gamma})$, the Heegaard Floer homology $HF^+(Y_{\gamma}(K),\mathfrak{s})$ is the homology of the mapping cone $\mathbb{X}^+_{\mathfrak{s}}$ of the chain map $D^+_{\mathfrak{s}}: \mathbb{A}^+_{\mathfrak{s}} \rightarrow \mathbb{B}^+_{\mathfrak{s}}$.
\end{theorem}

Ozsv\'{a}th and Szab\'{o} show that there exist grading shifts on $\mathbb{A}^+_{\mathfrak{s}}$ and $\mathbb{B}^+_{\mathfrak{s}}$, which gives a consistent relative $\mathbb{Z}$-grading on $\mathbb{X}^+_{\mathfrak{s}}$. Actually, the shift can be fixed such that the grading is the same as the absolute $\mathbb{Q}$-grading of $HF^+(Y_{\gamma}(K), \mathfrak{s})$. It is important to point out that these shifts only depend on the homology class of the knot.

Denote
\[\mathfrak{A}^+_{\xi}=H_{\ast}(A^+_{\xi}) \,\, ({\rm resp.} \,\, \mathfrak{B}^+_{\xi}=H_{\ast}(B^+_{\xi})), \qquad \mathfrak{A}^+_{\mathfrak{s}}=H_{\ast}(\mathbb{A}^+_{\mathfrak{s}}) \,\, ({\rm resp.} \,\, \mathfrak{B}^+_{\mathfrak{s}}=H_{\ast}(\mathbb{B}^+_{\mathfrak{s}})). \]
Let
\[\mathfrak{v}^+_{\xi}: \mathfrak{A}^+_{\xi} \rightarrow \mathfrak{B}^+_{\xi}, \;\;\mathfrak{h}^+_{\xi}: \mathfrak{A}^+_{\xi} \rightarrow \mathfrak{B}^+_{\xi+PD[\gamma]}\]
be the maps induced on homology by $v^+_{\xi}$ and $h^+_{\xi}$ respectively, and let
\[\mathfrak{D}^+_{\mathfrak{s}}: \mathfrak{A}^+_{\mathfrak{s}} \rightarrow \mathfrak{B}^+_{\mathfrak{s}}\]
be the map induced on homology by $D^+_{\mathfrak{s}}$. Theorem \ref{mappingconethm} implies the exact triangle
\begin{displaymath}
\xymatrix{
\mathfrak{A}^+_{\mathfrak{s}} \ar[r]^{\mathfrak{D}^+_{\mathfrak{s}}} & \mathfrak{B}^+_{\mathfrak{s}} \ar[d]^{{incl}_{\ast}} \\
 & HF^+(Y_{\gamma}(K),\mathfrak{s}) \ar[lu]^{{proj}_{\ast}}}
\end{displaymath}
Therefore, to compute either $HF^+(Y_{\gamma}(K),\mathfrak{s})$ or $d(Y_{\gamma}(K), \mathfrak{s})$, we study the kernel and cokernel of the map $\mathfrak{D}^+_{\mathfrak{s}}$.




Finally, we remark that there is an analogous mapping cone formula for the hat version of Heegaard Floer homology. One can define $\widehat{A}_{\xi}$, $\widehat{B}_{\xi}$, $\widehat{D}_{\xi}$ and the mapping cone $\widehat{\mathbb{X}}_{\mathfrak{s}}$ of $\widehat{D}_{\xi}$, and the Heegaard Floer homology $\widehat{HF}(Y_{\gamma}(K),\mathfrak{s})$ can be calculated by the homology of $\widehat{\mathbb{X}}_{\mathfrak{s}}$.

\subsection{Simple knots in lens spaces}

To compute the $d$-invariant of the surgered manifold $Y_{\gamma}(K)$, we need to fix the grading shift in the mapping cone formula. 
Since the grading shift only depends on the homology class of the knot, we may want to find it using a knot of the same homology class with simpler knot Floer complex. Simple knots in lens spaces will play such a role.  

For a lens space $L(p,q)$, there is a standard genus one Heegaard diagram (e.g., $L(5,1)$ in Figure \ref{simpleknot}), where we identify opposite sides of a rectangle to give a torus. We use a horizontal red curve to represent the $\alpha$ curve and use a blue curve of slope $p/q$ to represent the $\beta$ curve.  They intersect at $p$ points, $x_0, x_1, \dots , x_{p-1}$, where we label them in the order they appear on the $\alpha$ curve.  The $\alpha$ (resp. $\beta$) curve gives a solid torus $U_{\alpha}$ (resp. $U_{\beta}$).
\begin{figure}[!h]
\centering
\includegraphics[width=2.5in]{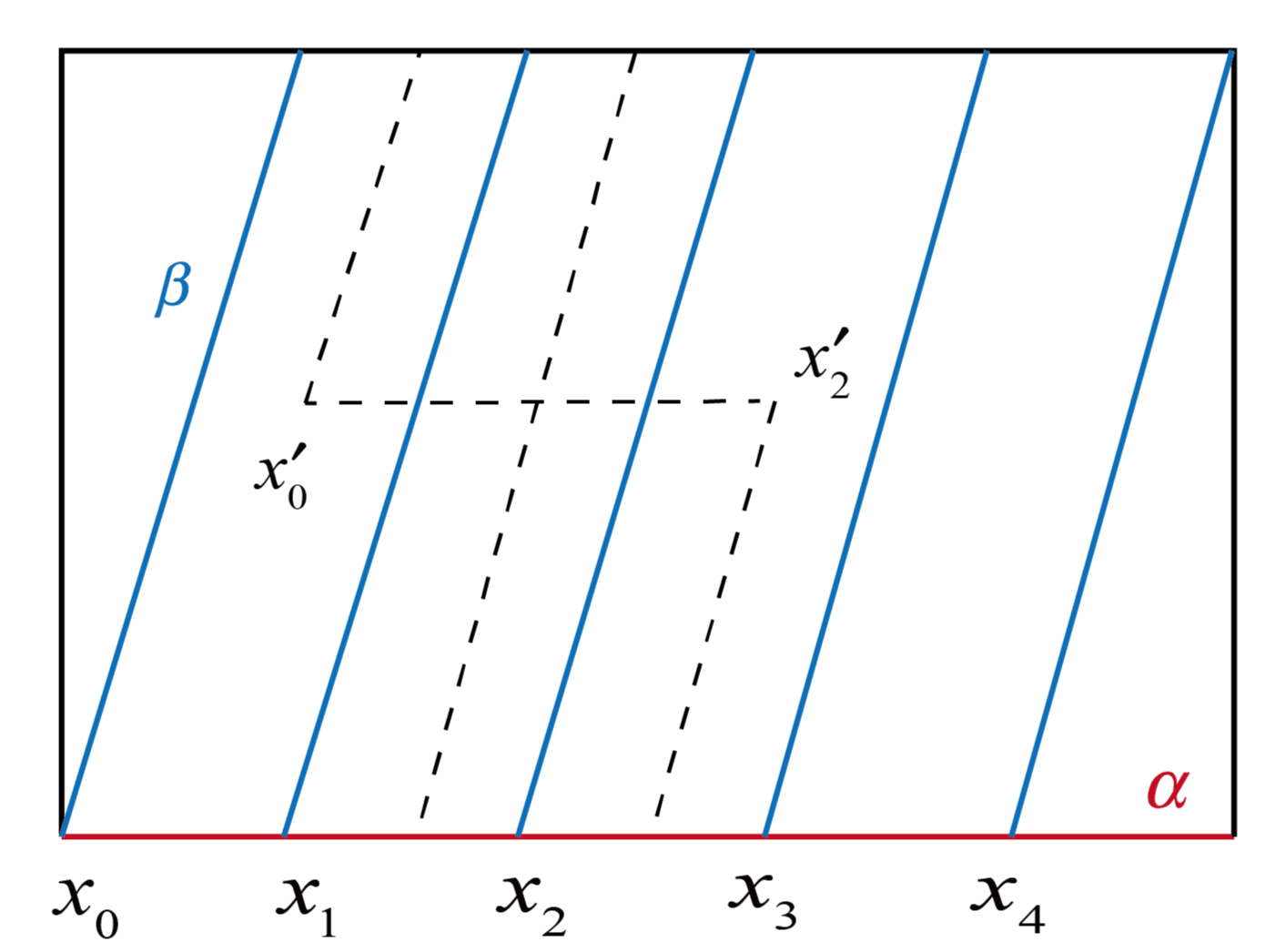}
\caption{An example of a simple knot $K(5,1,2)$ in $L(5,1)$}\label{simpleknot}
\end{figure}

\begin{definition}
The {\it simple knot} $K(p,q,k) \subset L(p,q)$ is an oriented knot defined as the union of the arc joining $x_0$ to $x_k$ in $U_{\alpha}$ and the arc joining $x_k$ to $x_0$ in $U_{\beta}$.
\end{definition}
To draw the simple knot $K(p,q,k)$ in the Heegaard diagram, we place two points $x'_0$ and $x'_k$ next to $x_0$ and $x_k$ respectively, and connect them in $U_{\alpha}$ and $U_{\beta}$, e.g., $K(5,1,2)$ in $L(5,1)$ illustrated in Figure \ref{simpleknot}.


In our case, the lens space is $L(p,1)$ for some prime number $p \geq 5$, and we consider simple knots $K(p,1,k)$ in $L(p,1)$. If we represent $K(p,1,k)$ in the standard genus one Heegaard diagram as described above, then the intersection points $x_0, \dots , x_{p-1}$ represent $p$ different $\rm Spin^c$ structures. Let $\eta(x_i, x_j)$ denote the one chain constructed by going from $x_i$ to $x_j$ along $\alpha$ curve and from $x_j$ to $x_i$ along $\beta$ curve. The relative Alexander grading of $x_i$ and $x_j$ is defined as
\[A(x_i, x_j)=[\eta(x_i, x_j)] \in H_1(L(p,1)-K(p,1,k)) \cong \mathbb{Z}.\]
We can fix the absolute Alexander gradings such that these values are symmetric about $0$. 

\begin{example}
Consider the simple knot $K(5,1,2)$ in Figure \ref{simpleknot}. One can check that $[\eta(x_3,x_4)]=[\eta(x_4,x_0)]=[\eta(x_0,x_1)]=-2x$ and $[\eta(x_1,x_2)]=[\eta(x_2,x_3)]=3x$ for a generator $x \in H_1(L(5,1)-K(5,1,2))\cong \mathbb{Z}$. Hence the absolute Alexander gradings of $x_3, x_4, x_2, x_0$ and $x_1$ are $-3$, $-1$, $0$, $1$ and $3$ respectively. Here, we fix the absolute Alexander grading by making it symmetric about $0$.
\end{example}

For a general simple knot $K(p,1,k)$ in $L(p,1)$, we have that $$[\eta(x_{k+1},x_{k+2})]=[\eta(x_{k+2},x_{k+3})]=\cdots=[\eta(x_{p-1},x_{0})]=[\eta(x_{0},x_{1})]=-kx$$
$$[\eta(x_{1},x_{2})]=\cdots=[\eta(x_{k},x_{k+1})]=(p-k)x,$$
where $x$ is a generator of $H_1(L(p,1)-K(p,1,k))$. Thus, the Alexander gradings of the $p$ points are
\begin{align}
&0, \pm (p-k), \pm 2(p-k), \dots, \pm \frac{k}{2} \cdot (p-k), \pm \frac{k}{2}, \pm \frac{3k}{2}, \dots, \pm (p-k-2) \cdot \frac{k}{2}   \quad {\rm when} \,\, k \,\,{\rm is \,\, even;} \label{eq28}\\
&0, \pm k, \pm 2k, \dots, \pm \frac{p-k}{2} \cdot k, \pm \frac{p-k}{2}, \pm \frac{3(p-k)}{2}, \dots, \pm (k-2) \cdot \frac{(p-k)}{2}  \quad {\rm when} \,\, k \,\,{\rm is \,\, odd.}\label{eq29}
\end{align}
Note that in either case, the largest and smallest Alexander grading are $\pm \frac{k}{2} \cdot (p-k)$ respectively.


For our purpose, we also introduce Rasmussen's notation for computing the hat version of the mapping cone formula \cite{Ras2}. We represent the chain complex $\widehat{\mathfrak{D}}_{\mathfrak{s}}: \widehat{\mathfrak{A}}_{\mathfrak{s}} \rightarrow \widehat{\mathfrak{B}}_{\mathfrak{s}}$ for a simple knot $K(p,1,k)$ by a type of diagram shown in Figure \ref{mapping cone of K(5,1,2) 1-surgery}: Here, the upper row of the diagram represents $\widehat{\mathfrak{A}}_{\xi}$, while the lower row of the diagram represents $\widehat{\mathfrak{B}}_{\xi}$. We denote $\widehat{\mathfrak{A}}_{\xi}$ by a $+$ if $\widehat{\mathfrak{v}}_{\xi}$ is nontrivial but $\widehat{\mathfrak{h}}_{\xi}$ is trivial, and we denote $\widehat{\mathfrak{A}}_{\xi}$ by a $-$ if  $\widehat{\mathfrak{h}}_{\xi}$ is nontrivial but  $\widehat{\mathfrak{v}}_{\xi}$ is trivial. Denote $\widehat{\mathfrak{A}}_{\xi}$ by a $\circ$ if both $\widehat{\mathfrak{v}}_{\xi}$ and $\widehat{\mathfrak{h}}_{\xi}$ are nontrivial. Each $\widehat{\mathfrak{B}}_{\xi}$ are represented by a filled circle. Nontrivial maps are indicated by arrows, and trivial maps are omitted.

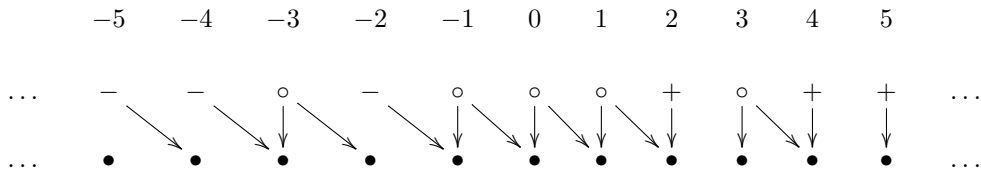
\begin{figure}[!h]
\begin{minipage}[b]{1\textwidth}
\centering
\begin{displaymath}
\xymatrixrowsep{5mm}
\xymatrixcolsep{5mm}
\xymatrix{
 & -5 & -4 & -3 & -2 & -1 & 0 & 1 & 2 & 3 & 4 & 5 & \\
\dots & - \ar[rd] & - \ar[rd] & \circ \ar[d] \ar[rd]& - \ar[rd] & \circ \ar[d] \ar[rd] & \circ \ar[d] \ar[rd] & \circ \ar[d] \ar[rd] & + \ar[d] & \circ \ar[d] \ar[rd] & + \ar[d] & + \ar[d] & \dots\\
\dots & \bullet & \bullet & \bullet  & \bullet & \bullet & \bullet & \bullet & \bullet & \bullet & \bullet & \bullet & \dots}
\end{displaymath}
\end{minipage}
\caption{The mapping cone of 1-surgery along $K(5,1,2)$. Note that all the elements on the left of -5 are $-$, and all the elements on the right of 5 are $+$.} 
\label{mapping cone of K(5,1,2) 1-surgery}
\end{figure}

The complex $\widehat{\mathfrak{D}}_{\mathfrak{s}}: \widehat{\mathfrak{A}}_{\mathfrak{s}} \rightarrow \widehat{\mathfrak{B}}_{\mathfrak{s}}$ can be decomposed into summands corresponding to the connected components of the diagram. For each summand, we denote it by an interval $[a,b]$, where $a$ and $b$ are labeled with a $+$ or $-$ and all the elements in between are $\circ$. We can see that summands of types $[+,+]$ and $[-,-]$ are acyclic and summands
of types $[-,+]$ and $[+,-]$ have homology of rank one. Moreover, when the summand is type $[-,+]$, the homology group $\mathbb{F}$ is supported by an element in the top row (i.e. in the kernel of $\widehat{\mathfrak{D}}_{\mathfrak{s}}$), and when the summand is type $[+,-]$, the homology group $\mathbb{F}$ is supported in the bottom row (i.e. the cokernel of $\widehat{\mathfrak{D}}_{\mathfrak{s}}$).





\subsection{The proof of the $d$-invariant surgery formula for homologically essential knots}

In this section, we will deduce our $d$-invariant surgery formula for homologically essential knots in $L(p,1)$. We split it into two cases $m>k^2/p$ and $m<k^2/p$ because the truncated mapping cones are different in the two cases. 

\begin{prop}
\label{prop ratinoal d-inv1}
Let $Y=L(p,1)$ with $p \geq 5$ prime and $K$ be a homologically essential knot in $Y$ with winding number $1 \leq k \leq \frac{p-1}{2}$. Suppose that $Y'$ is an L-space obtained from $\gamma=(m\mu+\lambda)$-surgery on $K$ with $m>k^2/p$, and $|H_1(Y')|\geq 5$ is odd.  If $p=5,7$ or $m \geq \frac{(p+k) \cdot k}{2p} +1$,  then there exists a non-negative integer $V_{\xi_0}$, a unique self-conjugate $\rm Spin^c$ structure $\mathfrak{t}$ on $Y'$ and a unique self-conjugate $\rm Spin^c$ structure $\mathfrak{t}_M$ on the Seifert fiber space $M(0,0;(m-k,1),(p-k,1),(k,1))$, abbreviated $M$, satisfying
\begin{equation}
\label{dinv1}
d(Y',\mathfrak{t})=d(M,\mathfrak{t}_M)-2V_{\xi_0}.
\end{equation}
If, in addition, $V_{\xi_0} \geq 2$, then there exists $V_{\xi_0+PD[\mu]}$ satisfying $V_{\xi_0}-1 \leq$ $V_{\xi_0+PD[\mu]} \leq V_{\xi_0}$ and
\begin{equation}
\label{dinv2}
d(Y',\mathfrak{t}+i^{\ast}PD[\mu])=d(M,\mathfrak{t}_M+i^{\ast}PD[\mu])-2V_{\xi_0+PD[\mu]},
\end{equation}
where $i: Y-K \rightarrow Y_{\gamma}(K)$ is inclusion.
\end{prop}

\begin{prop}
\label{prop ratinoal d-inv2}
Given $Y$, $K$ and $k$ as above, suppose that $Y'$ is an L-space obtained from $\gamma=(m\mu+\lambda)$-surgery on $K$ with $m<k^2/p$, and $|H_1(Y')| \neq 1$ is odd. If $p=5,7$ or $m \leq \frac{(3k-p) \cdot k}{2p} -1$, then there exists a non-negative integer $V_{\xi_0}$, a unique self-conjugate $\rm Spin^c$ structure $\mathfrak{t}$ on $Y'$ and a unique self-conjugate $\rm Spin^c$ structure $\mathfrak{t}_M$ on the Seifert fiber space $M(0,0;(m-k,1),(p-k,1),(k,1))$, abbreviated $M$, satisfying
\begin{equation}
\label{dinv3}
d(Y',\mathfrak{t})=d(M,\mathfrak{t}_M)+2V_{\xi_0}.
\end{equation}
If, in addition, $V_{\xi_0} \geq 2$, then there exists $V_{\xi_0+PD[\mu]}$ satisfying $V_{\xi_0}-1 \leq$ $V_{\xi_0+PD[\mu]} \leq V_{\xi_0}$ and
\begin{equation}
\label{dinv4}
d(Y',\mathfrak{t}+i^{\ast}PD[\mu])=d(M,\mathfrak{t}_M+i^{\ast}PD[\mu])+2V_{\xi_0+PD[\mu]},
\end{equation}
where $i: Y-K \rightarrow Y_{\gamma}(K)$ is inclusion.
\end{prop}

\begin{remark}
In our notation for the Seifert fiber space $M(0,0;(m-k,1),(p-k,1),(k,1))$, the two $0$'s means the base space for $M$ is of genus $0$ and without boundary, and $(m-k,1),(p-k,1)$, and $(k,1)$ specify the type of its exceptional fibers.
\end{remark}

Both propositions are deduced from the mapping cone formula. We will discuss the case $pm-k^2>0$ in detail, and the other case $pm-k^2<0$ can be obtained by reversing the orientation.

Fix a $\xi \in \underline{\rm Spin^c}(Y,K)=\mathbb{Z}$. Then the mapping cone $\mathbb{X}^+_{\mathfrak{s}}$ is given in Figure \ref{mappingcone}, where $\mathfrak{s}=G_{Y_{\gamma}(K), K_{\gamma}}(\xi)$. Since $$[\gamma]=m[\mu]+[\lambda]=(pm-k^2)[\theta],$$
the mapping cone $\mathbb{X}^+_{\mathfrak{s}}$ consists of $A^+_{\xi+j(pm-k^2) \cdot PD[\theta]}$ and $B^+_{\xi+j(pm-k^2) \cdot PD[\theta]}$ for $j \in \mathbb{Z}$.

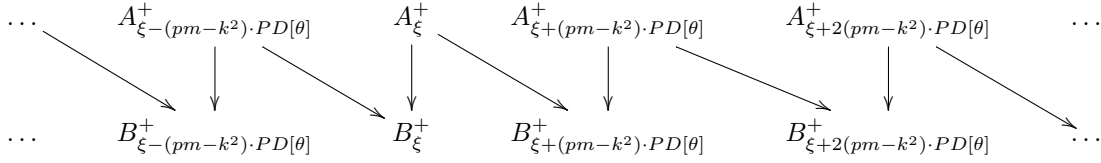
\begin{figure}[!h]
\begin{minipage}{1\linewidth}
\centering{\begin{displaymath}
\xymatrix{
\dots \ar[rd] & A^+_{\xi - (pm-k^2) \cdot PD[\theta]} \ar[d] \ar[rd] & A^+_{\xi} \ar[d] \ar[rd] & A^+_{\xi + (pm-k^2) \cdot PD[\theta]} \ar[d] \ar[rd] & A^+_{\xi + 2(pm-k^2) \cdot PD[\theta]} \ar[d] \ar[rd] & \dots\\
\dots  &  B^+_{\xi - (pm-k^2)\cdot PD[\theta]} & B^+_{\xi} &  B^+_{\xi + (pm-k^2)\cdot PD[\theta]} & B^+_{\xi + 2(pm-k^2)\cdot PD[\theta]} & \dots}
\end{displaymath}}
\end{minipage}
\caption{The mapping cone $\mathbb{X}^+_{\mathfrak{s}}$ with $\mathfrak{s}=G_{Y_{\gamma}(K), K_{\gamma}}(\xi)$.} \label{mappingcone}
\end{figure}

For a given $\xi$, there exists some positive integer $N$ such that $v^+_{\xi+i \cdot PD[\mu]}$ and $h^+_{\xi-i \cdot PD[\mu]}$ are quasi-isomorphisms when $i>N$. In the case $pm-k^2>0$, we see that $[\gamma]=(pm-k^2) \cdot [\theta]$ has the same sign as $[\mu]=p \cdot [\theta]$. Therefore the mapping cone $\mathbb{X}^+_{\mathfrak{s}}$ is quasi-isomorphic to the truncated mapping cone, denoted by $\mathbb{X}^{+,N}_{\mathfrak{s}}$, shown in Figure \ref{truncatedmappingcone}. 

\begin{figure}[!h]
\begin{minipage}{1\linewidth}
\centering{\begin{displaymath}
\xymatrix{
A^+_{\xi - N(pm-k^2) \cdot PD[\theta]} \ar[rd] & \dots \ar[d] \ar[rd] & A^+_{\xi} \ar[d] \ar[rd] & A^+_{\xi +(pm-k^2) \cdot PD[\theta]} \ar[d] \ar[rd] & \dots \ar[d] \ar[rd] & A^+_{\xi + N(pm-k^2) \cdot PD[\theta]} \ar[d] \\
& \dots & B^+_{\xi} & B^+_{\xi +(pm-k^2)\cdot PD[\theta]}& \dots  & B^+_{\xi + N(pm-k^2)\cdot PD[\theta]}}
\end{displaymath}}
\end{minipage}
\caption{The truncated mapping cone $\mathbb{X}^{+,N}_{\mathfrak{s}}$ with $\mathfrak{s}=G_{Y_{\gamma}(K), K_{\gamma}}(\xi)$ when $pm-k^2>0$.} \label{truncatedmappingcone}
\end{figure}
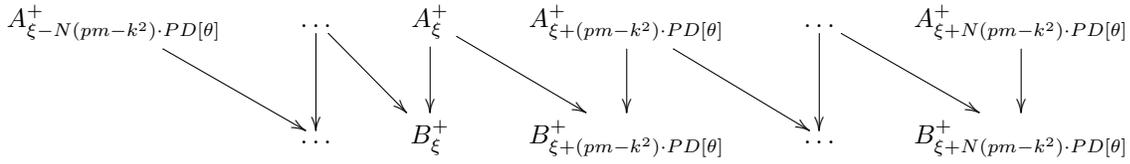

As $Y_{\gamma}(K)$ is an $L$-space obtained by a distance one surgery from an $L$-space, it follows from \cite[Lemma 6.7]{Boil} that
\[\widehat{\mathfrak{A}}_{\xi} \cong \mathbb{F}, \quad \mathfrak{A}^+_{\xi} \cong \mathcal{T}^+\]
for all $\xi \in \underline{\rm Spin^c}(Y,K)$. As $Y=L(p,1)$ itself is an $L$-space, we also have
\[\widehat{\mathfrak{B}}_{\xi} \cong \mathbb{F}, \quad \mathfrak{B}^+_{\xi} \cong \mathcal{T}^+ .\]
 This implies that $HF^+(Y_{\gamma}(K), \mathfrak{s})$ for any $\mathfrak{s} \in {\rm Spin^c}(Y_{\gamma}(K))$  is completely determined by the integers $V_{\xi}$ and $H_{\xi}$ for $\xi \in \underline{\rm Spin^c}(Y,K)$ with $G_{Y_{\gamma}(K),K_{\gamma}}(\xi)=\mathfrak{s}$.

Fix $n \gg 0$, and choose the parity of $n$ such that $|H_1(Y_{n\mu+\lambda})|=pn-k^2$ is odd. Then there exists only one self-conjugate $\rm Spin^c$ structure in ${\rm Spin^c}(Y_{n\mu+\lambda})$, denoted by $\mathfrak{t}_0$. For all sufficient large $n$, there is a map
\[\Xi: {\rm Spin^c}(Y_{\lambda+n\mu}(K)) \rightarrow \underline{\rm Spin^c}(Y,K). \]
Let $\xi_0=\Xi(\mathfrak{t}_0)$. The relative $\rm Spin^c$ structure $\xi_0$ has some key properties given in the following lemmas, which we can prove by the same arguments as in \cite{LMV}.

\begin{lemma}[Proposition 4.5 in \cite{LMV}]
\label{propsym}
Let $[l] \in H_1(Y-K)$. Then $V_{\xi_0+PD[l]}=H_{\xi_0-PD[l]}$.
\end{lemma}

\begin{lemma} [Lemma 4.7 in \cite{LMV}]
\label{lemmaselfconj}
The $\rm Spin^c$ structure $\mathfrak{s}_0=G_{Y_{\gamma}(K),K_{\gamma}}(\xi_0)$ is a self-conjugate $\rm Spin^c$ structure on $Y_{\gamma}(K)$.
\end{lemma}

\begin{lemma}[Lemma 4.8 in \cite{LMV}]
\label{lemmaxi0}
Let $\Pi^+_{\xi_0}: \mathbb{X}^+_{\mathfrak{s}_0} \rightarrow A^+_{\xi_0}$ be the natural quotient map, where $\mathfrak{s}_0=G_{Y_{\gamma}(K),K_{\gamma}}(\xi_0)$. Suppose that $Y_{\gamma}(K)$ is an $L$-space, then $\Pi^+_{\xi_0}$ is a quasi-isomorphism.
\end{lemma}

\begin{lemma}[Lemma 4.9 in \cite{LMV}]
\label{lemmaxi0+PDu}
Suppose that $Y_{\gamma}(K)$ is an $L$-space, and $V_{\xi_0} \geq 2$. Then the natural quotient map $\Pi^+_{\xi_0+PD[\mu]}: \mathbb{X}^+_{\mathfrak{s}_0 + i^{\ast} PD[\mu]} \rightarrow A^+_{\xi_0+PD[\mu]}$ is a quasi-isomorphism, where $i^{\ast}$ is induced by the inclusion $i: Y-K \rightarrow Y_{\gamma}(K)$.
\end{lemma}

\begin{proof}[Proof of Proposition \ref{prop ratinoal d-inv1}]
First, we use the truncated mapping cone $\mathbb{X}^+_{\mathfrak{s}_0}$ to show
\[d(Y',\mathfrak{s}_0)=d(M,\mathfrak{s}_0)-2V_{\xi_0}.\]
By Lemma \ref{lemmaselfconj}, $\mathfrak{s}_0$ is self-conjugate, and it is the unique self-conjugate $\rm Spin^c$ structure on $Y'$ and $M$ since $|H_1(Y')|=|H_1(M)|$ is assumed to be odd.

According to Lemma \ref{lemmaxi0}, the nonzero element of minimal grading in $HF^+(Y',\mathfrak{s}_0)$ is supported in $\mathfrak{A}^+_{\xi_0}$, so the minimal grading in $\mathfrak{A}^+_{\xi_0}$ is the $d$-invariant after an appropriate grading shift. Let $\sigma(\xi_0)$ denote this grading shift. Then we have
\begin{equation}
\label{eq12}
d(Y',\mathfrak{s}_0)=d(Y,G_{Y,K}(\xi_0))-2V_{\xi_0}+\sigma(\xi_0).
\end{equation}
Recall that grading shifts only depend on the homology class of the knot. So we use the simple knot in the same homology class as $K$, i.e., $K(p,1,k)$, to compute the grading shift. We can see that $\gamma$-surgery along $K(p,1,k)$ gives the Seifert fiber space $M(0,0;(m-k,1),(p-k,1),(k,1))$. This computation is standard (cf. \cite[Lemma 9]{Darcysumner}).

We claim that if $p=5, 7$ or $m \geq \frac{(p+k) \cdot k}{2p}$, then for the mapping cone $\mathbb{X}^+_{\mathfrak{s}_0}$ of $K(p,1,k)$, the nonzero element of minimal grading in $HF^+(M,\mathfrak{s}_0)$ is supported in $\mathfrak{A}^+_{\xi_0}$. We remark that one cannot directly apply Lemma \ref{lemmaxi0} here because the surgered manifold $M(0,0;(m-k,1),(p-k,1),(k,1))$ is not necessarily an $L$-space.  By Lemma \ref{propsym}, $\xi_0$ has the symmetric property $V_{\xi_0+PD[l]}=H_{\xi_0-PD[l]}$ for any $[l] \in H_1(Y-K)$. We see from (\ref{eq28}) and (\ref{eq29}) that the relative $\rm Spin^c$ structure with Alexander grading $0$ is the unique relative $\rm Spin^c$ structure which has this symmetric property, hence the relative $\rm Spin^c$ structure $\xi_{0}$ corresponds to $0$. Now we consider the hat version of the mapping cone.

When $m$ is large enough such that $pm-k^2 \geq \frac{k}{2} \cdot (p-k)$, where the right hand side is the largest Alexander grading of generators in the knot Floer complex of $K(p,1,k)$, the mapping cone is \emph{well-ordered}, i.e., there is one summand of type $[-,+]$ in the middle and all the elements of the top row to left of the $[-,+]$ summand are marked $-$, while all the elements to the right of the $[-,+]$ summand are marked $+$, as shown in Figure \ref{largemappingcone}. Hence the homology is isomorphic to $\mathbb{F}$, and $\widehat{\Pi}_{\xi_0}: \widehat{\mathbb{X}}_{\mathfrak{s}_0} \rightarrow \widehat{A}_{\xi_0}$ is a quasi-isomorphism.  It follows that $\Pi^+_{\xi_0}: \mathbb{X}^+_{\mathfrak{s}_0} \rightarrow A_{\xi_0}^+$ is also a quasi-isomorphism, so the nonzero element of minimal grading in $HF^+(M,\mathfrak{s}_0)$ is supported in $\mathfrak{A}^+_{\xi_0}$. 

\begin{figure}[!h]
\begin{minipage}[b]{1\textwidth}
\centering
\begin{displaymath}
\xymatrixrowsep{5mm}
\xymatrixcolsep{5mm}
\xymatrix{
\dots & - \ar[rd] &  - \ar[rd] & - \ar[rd] & \circ \ar[d] \ar[rd] & + \ar[d] & + \ar[d] & + \ar[d] &  \dots\\
\dots & \bullet & \bullet & \bullet & \bullet & \bullet & \bullet  &  \bullet  & \dots}
\end{displaymath}
\end{minipage}
\caption{The mapping cone of $K(p,1,k)$ when $m>\frac{(p+k) \cdot k}{2p}$. All the elements to the left of the summand $[-,+]$ are $-$, and all the elements to the right are $+$.
} \label{largemappingcone}
\end{figure}
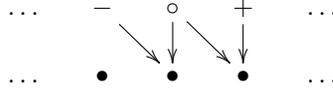

When $p=5$, $k=1$ and $m >0$, we can check that the mapping cone is also well-ordered. Hence, the minimal grading is also supported in $\mathfrak{A}_{\xi_0}^+$. When $p=5$ and $k=2$, the mapping cone for $m=1$ is shown in Figure \ref{mapping cone of K(5,1,2) 1-surgery} whereas the mapping cone for $m=2$ is well-ordered. In either case, we see that the minimal grading is supported in $\mathfrak{A}_{\xi_0}^+$.  Hence the claim is true when $p=5$. Finally, the case $p=7$ can be proved in a similar way, which we omit here.  


Once we understand where the nonzero element of minimal grading is supported, we can compute the $d$-invariant by the formula
\begin{equation}
\label{eq13}
d(M,\mathfrak{s}_0)=d(Y,G_{Y,K(p,1,k)}(\xi_0))+\sigma(\xi_0),
\end{equation}
where we use the fact that $V_{\xi_0}$ for the simple knot $K(p,1,k)$ equals 0.  Comparing (\ref{eq12}) and (\ref{eq13}), we obtain the desired equality
\[d(Y',\mathfrak{s}_0)=d(M,\mathfrak{s}_0)-2V_{\xi_0}.\]

The second equality (\ref{dinv2}) is proved by the same strategy. We use Lemma \ref{lemmaxi0+PDu} instead of Lemma \ref{lemmaxi0},  and under the assumption $|H_1(Y')|=pm-k^2\geq 5$ we can show that $\Pi^+_{\xi_0+PD[\mu]}$ for simple knot is a quasi-isomorphism when $p=5,7$ or $m \geq \frac{(p+k) \cdot k}{2p}+1$. Note that the inequality $m \geq \frac{(p+k) \cdot k}{2p}+1$ implies $pm-k^2 \geq \frac{k}{2}\cdot (p-k)+p$, which guarantees that the mapping cone of the simple knot is well-ordered. This completes the proof.  


\end{proof}

\begin{proof}[Proof of Proposition \ref{prop ratinoal d-inv2}]
The proof is similar to the case $pm-k^2>0$. 
We consider $-Y'$ as an $L$-space obtained from $(-m \mu + \lambda)$-surgery along a knot in $L(-p,1)$ and then apply the same argument. In particular, the inequality $m \leq \frac{(3k-p) \cdot k}{2p} -1 $ implies $pm-k^2 \leq -\frac{k}{2}\cdot(p-k)-p$, which guarantees that the mapping cone of the simple knot is well-ordered.
\end{proof}

\section{Surgeries along homologically essential knots}
\label{Surgeries along homologically essential knots}
\subsection{Distance one surgeries on $L(p,1)$ with $p \geq 5$ prime and $k=1$}
\label{L(p,1)}

We divide distance one surgeries along homologically essential knot in $L(p,1)$ into two cases: $k=1$ and $k>1$.
The case $k=1$ is simpler as the Seifert fiber space $M$ appearing in our $d$-invariant surgery formula is a lens space, of which the $d$-invariant is easier to compute.

More precisely, the Seifert fiber space $M=M(0,0; (m-k,1), (p-k,1), (k,1))$ in Proposition \ref{prop ratinoal d-inv1} and \ref{prop ratinoal d-inv2} reduces to the lens space $L(pm-1,p)$ when $k=1$. Note that the self-conjugate $\rm Spin^c$ structure $\mathfrak{t}_M$ in our $d$-invariant surgery formula corresponds to $\frac{p-1}{2}$, as it is the unique self-conjugate $\rm Spin^c$ structure on $L(pm-1,p)$. In addition, the $\rm Spin^c$ structure $\mathfrak{t}_M+i^{\ast}PD[\mu]$ in (\ref{dinv2}) and (\ref{dinv4}) corresponds to $\frac{3p-1}{2}$ up to $\rm Spin^c$-conjugation, because the difference of the two $\rm Spin^c$ structures $\frac{p-1}{2}$ and $\frac{3p-1}{2}$ is $\pm i^{\ast}PD[\mu]$.

For the convenience of the reader, we compute the $d$-invariant of the relevant lens spaces using the recursive formula (\ref{eq1}) and list the results below.

For $m \geq 2$,
\begin{align}
d(L(pm-1,1),0)&=\frac{pm-2}{4} \label{deq1}\\
d(L(pm-1,1),j)&=-\frac{1}{4}+\frac{(2j-(pm-1))^2}{4(pm-1)} \label{deq2}\\
d(L(pm-1,p),\frac{p-1}{2})&=\frac{m-2}{4} \label{deq3}\\
d(L(pm-1,p),\frac{3p-1}{2})&=\frac{pm^2-(6p+1)m+4p+6}{4(pm-1)}. \label{deq4}
\end{align}

For $m \leq -2$,
\begin{align}
d(L(-pm+1,1),0)&=\frac{-pm}{4} \label{deq5}\\
d(L(-pm+1,1),j)&=-\frac{1}{4}+\frac{(2j-(-pm+1))^2}{4(-pm+1)} \label{deq6}\\
d(L(-pm+1,p),\frac{p-1}{2})&=\frac{-m}{4} \label{deq7}\\
d(L(-pm+1,p),\frac{3p-1}{2})&=\frac{pm^2+(4p-1)m+4p-4}{4(-pm+1)}. \label{deq8}
\end{align}

\bigskip

\begin{proof}[Proof of Theorem \ref{theoremgen} (\romannumeral2)]
	Suppose the lens space $L(n,1)$ with $n$ odd is obtained by $(m\mu+\lambda)$-surgery along $K$. Since $|n|=pm-1$, $m$ is an even integer.  The proof is based on the computation of $d$-invariant, which we divide into 3 cases:
	
	\medskip
	\noindent
	Case \romannumeral1: $m>0$. \\
	By (\ref{firsthomology}), we have $|n|=pm-1$.
	
	If $n=pm-1$, Formula (\ref{dinv1}) gives
	\[d(L(pm-1,1),0)=d(L(pm-1,p),\frac{p-1}{2})-2V_{\xi_0},\]
	where we use the $\rm Spin^c$ structure corresponding to $0$ on $L(pm-1,1)$ as it is the unique self-conjugate $\rm Spin^c$ structure. Thus
	\[2V_{\xi_0}=\frac{m(1-p)}{4} <0,\]
	which contradicts the fact that $V_{\xi_0}$ is non-negative.
	
	If $n=-pm+1$, Formula (\ref{dinv1}) gives
	\[-d(L(pm-1,1),0)=d(L(pm-1,p),\frac{p-1}{2})-2V_{\xi_0},\]
	which implies
	\[V_{\xi_0}=\frac{m(p+1)-4}{8}.\]
	When $p=5$ and $m=2$, we have $V_{\xi_0}=1 < 2$, so Formula (\ref{dinv2}) is not applicable here. In this case, we cannot obstruct distance one surgery from $L(5,1)$ to $L(-9,1)$ by our $d$-invariant surgery formula. So that gives one of the possible solutions in the statement of Theorem \ref{theoremgen} (\romannumeral2).
	
	Otherwise, we have $p>5$ or $m \geq 4$, so the non-negative integer $V_{\xi_0}=\frac{m(p+1)-4}{8} \geq 2$. Applying Formula (\ref{dinv2}), we have
	\begin{equation}
	\label{eq15}
	-d(L(pm-1,1),0+i^{\ast}PD[\mu])=d(L(pm-1,p),\frac{3p-1}{2})-2V_{\xi_0+PD[\mu]}.
	\end{equation}
	Equation (\ref{deq2}) implies
	$-d(L(pm-1,1),0+i^{\ast}PD[\mu])= \frac{1}{4}-\frac{(2j-(pm-1))^2}{4(pm-1)}$
	for some $j \in [0,pm-2]$.  Plugging it into (\ref{eq15}), we have
	\begin{equation}
	\label{eq16}
	\frac{1}{4}-\frac{(2j-(pm-1))^2}{4(pm-1)}=\frac{pm^2-(6p+1)m+4p+6}{4(mp-1)}-2V_{\xi_0+PD[\mu]}
	\end{equation}
	Since $V_{\xi_0}-1 \leq$ $V_{\xi_0+PD[\mu]} \leq V_{\xi_0}$, there are two cases:

	Case \romannumeral1 (a): $V_{\xi_0+PD[\mu]}=V_{\xi_0}=\frac{m(p+1)-4}{8}$.
	Equation (\ref{eq16}) can be simplified to
	\[j^2-(pm-1)j+p+1-mp=0\]
	We claim that the function $f(x)=x^2-(pm-1)x+p+1-pm$ has no root in $[0,pm-2]$. Indeed, $f(x)<0$ for all $x \in [0,pm-2]$ since its axis of symmetry is $x=\frac{pm-1}{2}$ and $f(0)=p+1-mp<0$. So this case is impossible.
	
	Case \romannumeral1 (b): $V_{\xi_0+PD[\mu]}=V_{\xi_0}-1=\frac{m(p+1)-4}{8}-1$. Equation (\ref{eq16}) can be simplified to
	\[j^2-(pm-1)j+p+pm-1=0.\]
	Let $f(x)=x^2-(pm-1)x+p+pm-1$. The axis of symmetry of $f(x)$ is $x=\frac{pm-1}{2}$, and we can see that $f(0)=p+pm-1>0$, $f(1)=1+p>0$ and $f(2)=p+5-pm<0$. 
	Therefore, the roots of $f(x)$ lie in $(1,2)$ and $(pm-3,pm-2)$, which are not integers. This gives a contradiction.

	%
	
	\medskip\noindent
	Case \romannumeral2: $m<0$.\\ 
	By (\ref{firsthomology}), we have $|n|=-pm+1$. If $n=pm-1$, Formula (\ref{dinv3}) gives
	\[-d(L(-pm+1,1),0)+d(L(-pm+1,p),\frac{p-1}{2})=2V_{\xi_0},\]
	where we use the $\rm Spin^c$ structure corresponding to $0$ on $L(pm-1,1)$ as it is the unique self-conjugate $\rm Spin^c$ structure on it. Thus
	\[2V_{\xi_0}=\frac{m(p-1)}{4} <0,\]
	which contradicts the fact that $V_{\xi_0}$ is non-negative.
	If $n=-pm+1$, Formula (\ref{dinv3}) gives
	\[d(L(-pm+1,1),0)+d(L(-pm+1,p),\frac{p-1}{2})=2V_{\xi_0}.\]
	which implies the integer
	\[V_{\xi_0}=\frac{-m(p+1)}{8} \geq 2,\]
	since $m \leq -2$ and $p \geq 5$. Hence, we can apply Formula (\ref{dinv4}) and get
	\begin{equation}
	\label{eq17}
	d(L(-pm+1,1),0+i^{\ast}PD[\mu])+d(L(-pm+1,p),\frac{3p-1}{2})=2V_{\xi_0+PD[\mu]}.
	\end{equation}
	By Equation (\ref{deq6}),
	$d(L(-pm+1,1),0+i^{\ast}PD[\mu])=-\frac{1}{4}+\frac{(2j-(-pm+1))^2}{4(-pm+1)}$
	for some $j \in [0,-pm]$.  Plugging it into (\ref{eq17}), we have
	\begin{equation}
	\label{eq18}
	-\frac{1}{4}+\frac{(2j-(-pm+1))^2}{4(-pm+1)} +\frac{pm^2+(4p-1)m+4p-4}{4(-mp+1)}=2V_{\xi_0+PD[\mu]}
	\end{equation}
	Since $V_{\xi_0}-1 \leq V_{\xi_0+PD[\mu]} \leq V_{\xi_0}$, there are two cases:
	
	Case \romannumeral2 (a): $V_{\xi_0+PD[\mu]}=V_{\xi_0}=\frac{-m(p+1)}{8}$. Equation (\ref{eq18}) can be simplified to
	\[j^2-(-pm+1)j+p+pm-1=0\]
	We claim that the function $f(x)=x^2-(-pm+1)x+p+pm-1$ has no root in $[0,-pm]$. Indeed, $f(x)<0$ for any $x \in [0,-pm]$ since its axis of symmetry is $x=\frac{-mp+1}{2}$ and $f(0)=p+pm-1<0$. So this case is impossible.
	
	Case \romannumeral2 (b): $V_{\xi_0+PD[\mu]}=V_{\xi_0}-1=\frac{-m(p+1)}{8}-1$. Equation (\ref{eq18}) can be simplified to
	\[j^2-(-pm+1)j+p-pm+1=0\]
	Let $f(x)=x^2-(-pm+1)x+p-pm+1$. The axis of symmetry of $f(x)$ is $x=\frac{-pm+1}{2}$, and $f(0)=p-pm+1>0$, $f(1)=p+1>0$ and $f(2)=p+3+pm<0$. So the roots of $f(x)$ lie in the intervals $(1,2)$ and $(-pm-1,-pm)$, which are not integers. This gives a contradiction.
	%
	
	\medskip
	\noindent
	Case \romannumeral3: $m=0$. By (\ref{firsthomology}), we have $|n|=1$. In this case, we indeed have a distance one surgery on $L(p,1)$ yielding $S^3$ given as the double branched cover of the band surgery in Figure \ref{bandsurgery4}.

\end{proof}

\subsection{$d$-invariants of Seifert fiber spaces}

To use the $d$-invariant surgery formula for homologically essential knots in $L(p,1)$, we must compute the $d$-invariant of the Seifert fiber space $M(0,0;(m-k,1),(p-k,1),(k,1))$. In this subsection, we briefly review the algorithm due to Ozsv\'{a}th and Szab\'{o}, which computes the $d$-invariant of a larger class of 3-manifolds, namely, the plumbed 3-manifolds \cite{OSd}.

Let $N$ be a Seifert fiber space. Then $N$ can be regarded as the boundary of a plumbed 4-manifold $X$, which is constructed by plumbing disc bundles over $S^2$ according to a diagram $G$. For example, in our case $N=M(0,0;(m-k,1),(p-k,1),(k,1))$,  and the plumbing diagram $G$ is shown in Figure \ref{Mplumbinggraph}.

\begin{figure}[!h]
\centering
\subfigure[When $m \geq k+3$.]{\label{M1}
\includegraphics[width=0.4\textwidth]{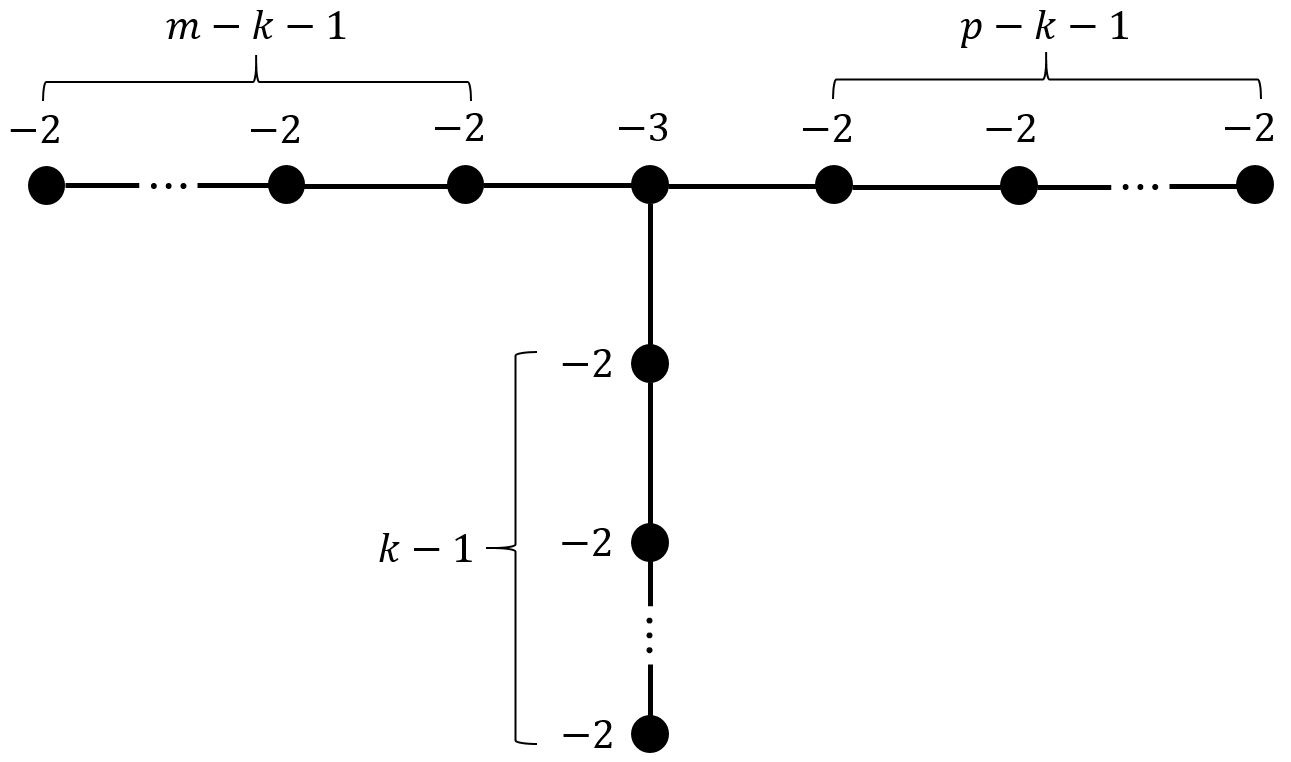}}\qquad
\subfigure[When $m \leq k-3$.]{\label{M2}
\includegraphics[width=0.4\textwidth]{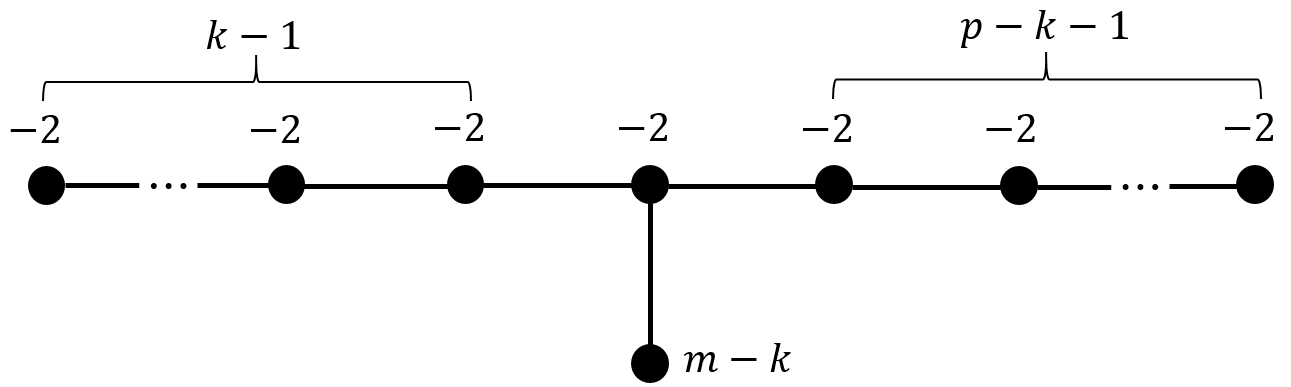}}
\caption{The plumping diagram of $M(0,0;(m-k,1),(p-k,1),(k,1))$ when $k>1$.}\label{Mplumbinggraph}
\end{figure}

Suppose $N$ is a rational homology sphere.  We have the following exact sequence.  

\begin{displaymath}
\xymatrix{
   & & {\rm Spin^c}(X) \ar[d]^{c_1} \ar[r] & {\rm Spin^c}(N) \ar[d]^{c_1} & \\
0 \ar[r] & H_2(X) \ar[r] \ar@{=}[d] & H^2(X) \ar[r] \ar@{=}[d] & H^2(N) \ar[r] \ar@{=}[d] & 0\\
0 \ar[r] & \mathbb{Z}^{b_2(X)} \ar[r]^Q & \mathbb{Z}^{b_2(X)} \ar[r] & {\rm coker}Q \ar[r] & 0}
\end{displaymath}

Let $V(G)$ denote the set of vertices of $G$. For each $v \in V(G)$ let $S_v$ be the sphere corresponding to $v$. The homology $H_2(X)$ is free and generated by the homology class of spheres $S_v$. Since $X$ is simply-connected, the cohomology $H^2(X)=\mathrm{Hom}(H_2(X),\mathbb{Z})$. As $H^2(X)$ is free, it can be represented over the basis $[S_v]^{\ast}$, the Hom-dual of $[S_v]$. Over the basis $[S_v]$ and $[S_v]^{\ast}$, the map $H_2(X) \rightarrow H^2(X)$ is represented by the matrix of the intersection form $Q: H_2(X) \times H_2(X) \rightarrow \mathbb{Z}$.

The set of characteristic vectors for $G$, denoted by Char(G), consists of those $w \in H^2(X)$ satisfying
\[\langle w , [S_v] \rangle \equiv \langle [S_v], [S_v] \rangle \,\, ({\rm mod} \,\, 2) \]
for all $v \in V(G)$. The set of $\rm Spin^c$ structures on $X$ is in one-to-one correspondence with the characteristic vectors for $G$ by the first Chern class $c_1$. We use $\mathfrak{t}(w)$ to represent the $\rm Spin^c$ structure on $N$ that is determined by the equivalence class $w \in $ Char(G) in $\mathrm{coker}\,Q$. If $\mathfrak{s}, \mathfrak{s}' \in {\rm Spin^c}(X)$ restrict to the same ${\rm Spin^c}$ structure on $N$, then their corresponding characteristic vectors $w, w' \in$ Char(G) are congruent modulo the image of $2H_2(X)$ in $H^2(X)$; equivalently, $(w-w')Q^{-1} \in H_2(X)$.

Denote by $\omega(v)$ and $d(v)$ the \emph{weight} and \emph{degree} of a vertex $v$ in $G$, respectively. The vertex $v$ is called bad if $\omega(v) > -d(v)$.  Ozsv\'{a}th and Szab\'{o} \cite{OSd} show that if $G$ is negative definite and contains at most one bad vertex, then
\begin{equation}
\label{dSFS}
d(N,\mathfrak{t})=\frac{1}{4} \left( \mathop{\max}_{w:\mathfrak{t}(w)=\mathfrak{t}} \langle w,w \rangle + |G| \right).
\end{equation}
Moreover, they give an algorithm to find the characteristic covector $w$ that maximises (\ref{dSFS}), which we review below. 

Consider all $w \in$ Char(G) satisfying
\begin{equation}
\label{eq14}
\omega(v)+2 \leq \langle w, [S_v] \rangle \leq -\omega(v) \,\, {\rm for \,\, all \,\, } v \in V(G).
\end{equation}
Let $w_0=w$. We then construct $w_i$ inductively as follows: if there exists $v_j \in V(G)$ such that
\[ \langle w_i, [S_{v_j}] \rangle = -\omega(v_j), \]
then we let $w_{i+1}=w_{i}+2PD[S_{v_j}]$ and call this action a {\it pushing down} the value of $w_i$ on $v_j$. 
The path $\{ w_0, w_1, \dots \}$ will terminate at some $w_n$ when one of the followings happens:
\begin{itemize}
\item $\omega(v) \leq \langle w_n , [S_v] \rangle \leq -\omega(v)-2$ for all $v \in V(G)$. In this case, the path is called {\it maximising}.
\item $\langle w_n, [S_v] \rangle > -\omega(v)$ for some $v \in V(G)$. In this case, the path is called {\it non-maximising}.
\end{itemize}
Ozsv\'{a}th and Szab\'{o} proved that the maximiser of (\ref{dSFS}) is contained in the set of characteristic vectors which satisfy (\ref{eq14}) and initiate a maximising path.

\subsection{Distance one surgeries on $L(p,1)$ with $p \geq 5$ prime and $k>1$}

The goal of this section is to prove Theorem \ref{theoremgen}(\romannumeral3): We want to show that a distance one surgery never yields a lens space $L(n,1)$ when $m\geq k+3$.  We achieve this by showing that the $d$-invariant of the lens space $L(n,1)$ never equals the value obtained from our $d$-invariant surgery formula for homologically essential knots.

The first step is to compute the $d$-invariant of Seifert fiber space $M(0,0;(m-k,1),(p-k,1),(k,1))$. When $m\geq k+3$, the negative-definite plumbing diagram of $M$ on which we apply Ozsv\'ath-Szab\'o's formula (\ref{dSFS}) is given in Figure \ref{M1}. We compute the intersection form $Q$ associated to this plumbing diagram:



\[Q_M=\left[\begin{array}{c|c|c|c}
\overbrace{\begin{matrix}
 -2 & 1 &       &   \\
 1  & -2&       &   \\
    &   &\ddots & 1 \\
    &   &      1& -2
\end{matrix}}^{m-k-1}& & &
\begin{matrix}
\\
\\
\\
1
\end{matrix}\\
\hline
& \overbrace{\begin{matrix}
 -2 & 1 &       &   \\
 1  & -2&       &   \\
    &   &\ddots & 1 \\
    &   &      1& -2
\end{matrix}}^{p-k-1}& &
\begin{matrix}
\\
\\
\\
1
\end{matrix}\\
\hline
& &
\overbrace{\begin{matrix}
 -2 & 1 &       &   \\
 1  & -2&       &   \\
    &   &\ddots & 1 \\
    &   &      1& -2
\end{matrix}}^{k-1}&
\begin{matrix}
\\
\\
\\
1
\end{matrix}\\
\hline
\hspace{2.5cm} 1 & \hspace{2.5cm} 1 & \hspace{2.5cm} 1 & -3
\end{array}\right]\]

\medskip
The next lemma gives the maximiser of Formula (\ref{dSFS}) for each $\rm Spin^c$ structure on $M$.

\begin{lemma}
\label{maximiser}
The maximisers for the $pm-k^2$ number of $\rm Spin^c$ structures on $M=M(0,0;(m-k,1),(p-k,1),(k,1))$ with $m \geq k+3$ and $k>1$ are given as follows:

\begin{enumerate}
\item $w^1_2(i,j)=(0, \dots, 0,2,0, \dots, 0 \arrowvert 0, \dots, 0,2,0, \dots, 0 \arrowvert 0, \dots, 0 \arrowvert -1)$. Here, $1 \leq i \leq m-k-1$ and $1 \leq j \leq p-k-1$ denote the place where $2$ appears (e.g., $w^1_2(1,1)=(2, \dots, 0 \arrowvert 2, \dots, 0 \arrowvert 0, \dots, 0 \arrowvert -1)$).\\

    \noindent $w^2_2(i,j)=(0, \dots, 0,2,0, \dots, 0 \arrowvert 0, \dots, 0  \arrowvert 0, \dots, 0,2,0, \dots, 0  \arrowvert -1)$. Here, $1 \leq i \leq m-k-1$ and $1 \leq j \leq k-1$ denote the place where $2$ appears.\\

    \noindent $w^3_2(i,j)=(0, \dots, 0 \arrowvert 0, \dots, 0,2,0, \dots, 0 \arrowvert 0, \dots, 0,2,0, \dots, 0  \arrowvert -1)$. Here $1 \leq i \leq p-k-1$ and $1 \leq j \leq k-1$ denote the place where $2$ appears.\\

\item $w^1_1(i,\pm 1)=(0, \dots, 0,2,0, \dots, 0 \arrowvert 0, \dots, 0 \arrowvert 0, \dots, 0 \arrowvert \pm 1)$. Here, $1 \leq i \leq m-k-1$ denotes the place where $2$ appears; the last element can be either $+1$ or $-1$.\\

    \noindent $w^2_1(i, \pm 1)=(0, \dots, 0 \arrowvert  0, \dots, 0,2,0, \dots, 0 \arrowvert 0, \dots, 0 \arrowvert \pm 1)$.  Here, $1 \leq i \leq p-k-1$ denotes the place where $2$ appears; the last element can be either $+1$ or $-1$.\\

    \noindent $w^3_1(i, \pm 1)=(0, \dots, 0 \arrowvert 0, \dots, 0 \arrowvert 0, \dots, 0,2,0, \dots, 0 \arrowvert \pm 1)$. Here, $1 \leq i \leq k-1$ denotes the place where $2$ appears; the last element can be either $+1$ or $-1$.\\

    \noindent  $w^4_1=(0, \dots, 0 \arrowvert 0, \dots, 0 \arrowvert 0, \dots, 0 \arrowvert 3)$.\\

\item $w^1_0=(0, \dots, 0  \arrowvert 0, \dots, 0 \arrowvert 0, \dots, 0 \arrowvert 1)$. \\

\noindent $w^2_0=(0, \dots, 0  \arrowvert 0, \dots, 0 \arrowvert 0, \dots, 0 \arrowvert -1)$.
\end{enumerate}
In the above notation, we divide vectors by vertical bars to 4 blocks, which contain $m-k-1$, $p-k-1$, $k-1$, and $1$ elements, respectively. The subscripts in $w$ stand for the number of $2$ in the corresponding vector, and the superscripts are used to distinguish different type of those vectors which contain the same number of $2$.
\end{lemma}

\begin{proof}
We see that there is no bad vertex in the plumbing diagram of $M$. It follows from \cite[Lemma 2.7, Proposition 3.2]{OSd} that the number of characteristic vectors which satisfy (\ref{eq14}) and initiate a maximising path must equal the number of $\rm Spin^c$ structures over $M$.

Given a vector $w$ satisfying (\ref{eq14}), suppose $w$ contains a substring $(2,0, \dots, 0,2)$ in one of the blocks in the above vector notation.  When we push down the 2's from left to right in the substring, we will eventually obtain a $4$ at the last spot of the substring. Thus, we conclude that if there exist two vertices $v \in V(G)$ in the same block satisfying $\langle w , [S_v] \rangle=-\omega(v)$, then $w$ initiates a non-maximising path.   
So for a maximiser $w$, there are at most four $v \in V(G)$ such that $\langle w , [S_v] \rangle=-\omega(v)$; because otherwise, the pigeonhole principle implies that there must be two of $v \in V(G)$ in the same block.

Now we consider the following 5 cases.

(1) There are four $v \in V(G)$ such that $\langle w , [S_v] \rangle=-\omega(v)$. Then the vector looks like $(0, \dots, 0,2,\\0, \dots, 0 \arrowvert 0, \dots, 0,2,0, \dots, 0  \arrowvert 0, \dots, 0,2,0, \dots, 0  \arrowvert 3)$. Pushing down the last element 3, we will have two 2's in each of the first 3 blocks. So this $w$ initiates a non-maximising path.

(2) There are three $v \in V(G)$ such that $\langle w , [S_v] \rangle=-\omega(v)$. If $w$ has two 2's in two of the first 3 blocks and a 3 in the last block, then similar to (1) we will eventually get two 2's in one block after pushing down the 3. If $w$ looks like $(0, \dots, 0,2,0, \dots, 0 \arrowvert 0, \dots, 0,2,0, \dots, 0  \arrowvert 0, \dots, 0,2,0, \dots, 0  \arrowvert j)$ with $-1 \leq j \leq 1$, then after we push down the 2's, the last element $j$ will change to $j+6 >3$.

(3) There are two $v \in V(G)$ such that $\langle w , [S_v] \rangle=-\omega(v)$. If $w$ has a 3 in the last block and a 2 in one of the first 3 blocks, then similar to (1) we will eventually get two 2's in one block. If $w$ has two 2's, then after we push down the 2's, the last element $j$ will change to $j+4$. So $j$ can only be -1. By doing a similar pushing down, this shows that vectors of type $w^1_2$, $w^2_2$ and $w^3_2$ initiate maximising paths.


(4) There are one $v \in V(G)$ such that $\langle w , [S_v] \rangle=-\omega(v)$. Similar argument implies that only the vectors of type $w^1_1, w^2_1, w^3_1, w^4_1$ initiate maximising paths.

(5) There is no $v \in V(G)$ such that $\langle w , [S_v] \rangle=-\omega(v)$. In fact, only $w^1_0$ and $w^2_0$ are vectors of such type, and both of them initiate maximising paths.

In summary, we have found a total number of $pm-k^2$ vectors which agrees with the number of $\rm Spin^c$ structures over $M$. Each of them is a maximiser of Formula (\ref{dSFS}) for the corresponding $\rm Spin^c$ structure.


\end{proof}

The next goal is to compute $d(M,\mathfrak{t}_M)$ and $d(M,\mathfrak{t}_M + i^{\ast} PD[\mu])$, where $\mathfrak{t}_M$ is the unique self-conjugate $\rm Spin^c$ structure on $M$. The following lemma determines the corresponding maximisers.  

\begin{lemma}
\label{correpondingcovoter}
Let $M=M(0,0;(m-k,1),(p-k,1),(k,1))$ with $m \geq k+3$ and $k>1$.
\begin{enumerate}
\item When $k$ is even, the characteristic vector $w^3_1(\frac{k}{2},-1)=(0, \dots, 0 \arrowvert 0, \dots, 0 \arrowvert 0, \dots, 0,2,0, \dots, 0 \arrowvert -1)$ (resp. $w^2_2(1,\frac{k}{2})=(2, \dots, 0 \arrowvert 0, \dots, 0 \arrowvert 0, \dots, 0,2,0, \dots, 0 \arrowvert -1)$) is the maximiser of Formula (\ref{dSFS}) in the equivalence class,  which corresponds to the $\rm Spin^c$ structure $\mathfrak{t}_M$ (resp. $\overline{\mathfrak{t}_M + i^{\ast} PD[\mu]}$).
\item When $k$ is odd, the characteristic vector $w^2_1(\frac{p-k}{2},-1)=(0, \dots, 0 \arrowvert 0, \dots, 0,2,0, \dots, 0 \arrowvert 0, \dots, 0 \arrowvert -1)$ (resp. $w^1_2(1,\frac{p-k}{2})=(2, \dots, 0 \arrowvert 0, \dots, 0,2,0, \dots, 0  \arrowvert 0, \dots, 0 \arrowvert -1)$) is the maximiser of Formula (\ref{dSFS}) in the equivalence class, which corresponds to the $\rm Spin^c$ structure $\mathfrak{t}_M$ (resp. $\overline{\mathfrak{t}_M + i^{\ast} PD[\mu]}$).
\end{enumerate}
\end{lemma}

\begin{proof}
Since the first Chern class $c_1$ of a self-conjugate $\rm Spin^c$ structure is 0, the unique self-conjugate $\rm Spin^c$ structure $\mathfrak{t}_M$ must correspond to the equivalence class of characteristic vectors which are in the image of $Q_M$.

When $k$ is even, we can see that among all maximisers, only $w^3_1(\frac{k}{2},-1)$ satisfies this; more precisely, $w^3_1(\frac{k}{2},-1)=Q_M v^T$ where $v=(0, \dots, 0 \arrowvert 0, \dots, 0 \arrowvert -1, -2, \dots, -\frac{k}{2}+1, -\frac{k}{2}, -\frac{k}{2}+1, \dots, -2, -1 \arrowvert 0)$.

When $k$ is odd, only $w^2_1(\frac{p-k}{2},-1)$ is in the image of $Q_M$; more precisely, $w^2_1(\frac{p-k}{2},-1)=Q_M v^T$, where $v=(0, \dots, 0 \arrowvert -1, -2, \dots, -\frac{p-k}{2}+1, -\frac{p-k}{2}, -\frac{p-k}{2}+1, \dots, -2, -1 \arrowvert 0, \dots, 0 \arrowvert 0)$.

To find the maximiser corresponding to $\mathfrak{t}_M + i^{\ast} PD[\mu]$, we represent $M$ as the boundary of a plumbed 4-manifold $X$ given by the framed link $\mathbb{L}=((K_1,m_1), \dots, (K_{p+m-k-2}, m_{p+m-k-2}))$ in Figure \ref{framinglink}. Denote by $D_i$ a small normal disk to $K_i$, and $\partial D_i=\mu_i$ the meridian of $K_i$. A systematic yet strenuous computation of homology shows that $\mu=-\mu_1$.  Thus, $i^{\ast}(PD[\mu])$ corresponds to $-PD[D_1] \in H^2(X)$ which is represented by the vector $(-1,0,\dots,0 \arrowvert 0, \dots, 0 \arrowvert 0, \dots, 0 \arrowvert 0)$. Therefore, $\overline{\mathfrak{t}_M + i^{\ast} PD[\mu]}$ corresponds to the characteristic vector $w^2_2(1,\frac{k}{2})$ (resp. $w^1_2(1,\frac{p-k}{2})$) when $k$ is even (resp. odd).

\begin{figure}[!h]
\centering
\includegraphics[width=6in]{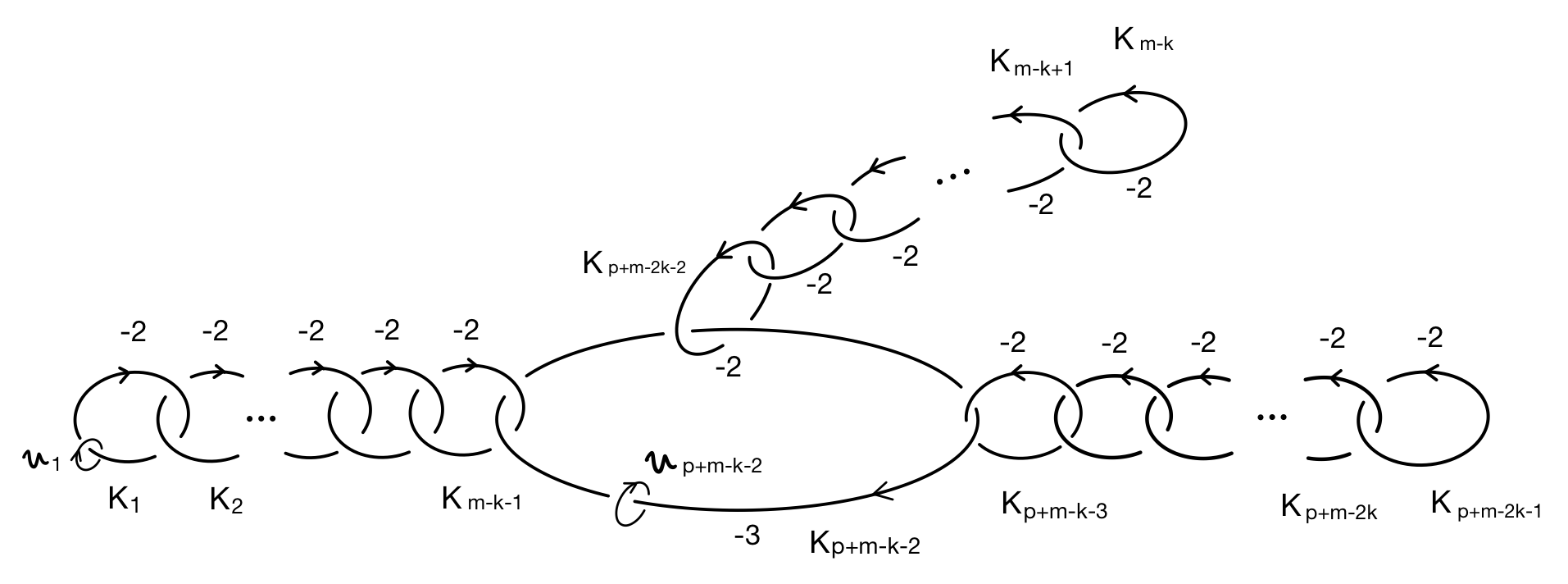}
\caption{The surgery description $\mathbb{L}=((K_1,m_1), \dots, (K_{p+m-k-2}, m_{p+m-k-2}))$ of the plumbed 4-manifold $X$ with $\partial X=M$. Denote by $D_i$ a small normal disk to $K_i$, and $\partial D_i=\mu_i$ the meridian of $K_i$.}
\label{framinglink}
\end{figure}




\end{proof}

Now we can use Formula (\ref{dSFS}) to compute $d(M,\mathfrak{t}_M)$ and $d(M,\mathfrak{t}_M + i^{\ast} PD[\mu])$.

\begin{itemize}
  \item When $k$ is even,
  \begin{align*}
  d(M,\mathfrak{t}_M)&=\frac{m+p-2k-2}{4}\\
  d(M,\mathfrak{t}_M + i^{\ast} PD[\mu])&=\frac{pm^2-(6p+2kp-p^2+k^2)m+4p+6k^2+2k^3-pk^2}{4(pm-k^2)}
  \end{align*}
  \item When $k$ is odd,
  \begin{align*}
  d(M,\mathfrak{t}_M)&=\frac{m-2}{4}\\
  d(M,\mathfrak{t}_M + i^{\ast} PD[\mu])&=\frac{pm^2-(6p+k^2)m+6k^2+4p}{4(pm-k^2)}
  \end{align*}
\end{itemize}


\bigskip

\begin{proof}[Proof of Theorem \ref{theoremgen} (\romannumeral3)]
Suppose $L(n,1)$ with $n$ odd is obtained by $(m \mu+\lambda)$-surgery along $K$ for some $m \geq k+3$. Then $m \geq (\frac{p+k}{2p}) \cdot k +1 > \frac{k^2}{p}$. So we can apply the $d$-invariant surgery formula in Proposition \ref{prop ratinoal d-inv1}. We divide the computation into 2 cases:

\medskip\noindent
Case \romannumeral1: $k$ is even. By (\ref{firsthomology}), $|n|=pm-k^2$.

If $n=pm-k^2$, Formula (\ref{dinv1}) implies
\[d(L(pm-k^2,1),0)=d(M,\mathfrak{t}_M)-2V_{\xi_0}.\]
Then
\[V_{\xi_0}=\frac{p+k^2-2k-1-(p-1)m}{8} \leq \frac{p+k^2-2k-1-(p-1)(k+3)}{8}=\frac{k^2+2-kp-2p-k}{8}<0,\]
which contradicts the fact that $V_{\xi_0} \geq 0$.

If $n=-pm+k^2$, Formula (\ref{dinv1}) gives
\[-d(L(pm-k^2,1),0)=d(M,\mathfrak{t}_M)-2V_{\xi_0},\]
which implies
\[V_{\xi_0}=\frac{(p+1)m+p-2k-k^2-3}{8} \geq \frac{(p+1)(k+3)+p-2k-k^2-3}{8}= \frac{pk+4p-k(k+1)}{8} \geq 2.\]
We can thus apply Formula (\ref{dinv2}) and get
\[-d(L(pm-k^2,1), 0+i^{\ast}PD[\mu])=d(M,\mathfrak{t}_M+i^{\ast}PD[\mu])-2V_{\xi_0+PD[\mu]}.\]
Here, $-d(L(pm-k^2,1),0+i^{\ast}PD[\mu])= \frac{1}{4}-\frac{(2j-(pm-k^2))^2}{4(pm-k^2)}$
for some $j \in [0,pm-k^2-1]$. Therefore,
\begin{equation}
\label{eq19}
\frac{1}{4}-\frac{(2j-(pm-k^2))^2}{4(pm-k^2)}=\frac{pm^2-(6p+2kp-p^2+k^2)m+4p+6k^2+2k^3-pk^2}{4(pm-k^2)}-2V_{\xi_0+PD[\mu]}
\end{equation}
Since $V_{\xi_0}-1 \leq V_{\xi_0+PD[\mu]} \leq V_{\xi_0}$, we further divide it into two cases:

Case \romannumeral1(a): $V_{\xi_0+PD[\mu]}=V_{\xi_0}=\frac{(p+1)m+p-2k-k^2-3}{8}$. Equation (\ref{eq19}) can be simplified to
\begin{equation}
\label{eq26}
j^2-(pm-k^2)j+p+k^2-mp=0.
\end{equation}
We claim that the function $f(x)=x^2-(pm-k^2)x+p+k^2-mp$ has no root in $[0,pm-k^2-1]$. Indeed, $f(x)<0$ for any $x \in [0,pm-k^2-1]$ since its axis of symmetry is $x=\frac{pm-k^2}{2}$ and $f(0)=p+k^2-mp \leq p+k^2-3p-kp <0$. So this case is impossible.

\medskip
Case \romannumeral1(b): $V_{\xi_0+PD[\mu]}=V_{\xi_0}-1=\frac{(p+1)m+p-2k-k^2-3}{8}-1$. Equation (\ref{eq19}) can be simplified to
\begin{equation}
\label{eq27}
j^2-(pm-k^2)j+p+pm-k^2=0.
\end{equation}
Let $f(x)=x^2-(pm-k^2)x+p+pm-k^2$. The axis of symmetry of $f(x)$ is $x=\frac{pm-k^2}{2}$, and $f(0)=p+pm-k^2>0$, $f(1)=p+1>0$ and $f(2)=p+4-pm+k^2 \leq p+4+k^2-3p-kp <0$. Therefore, the roots of $f(x)$ lie in $(1,2)$ and $(pm-k^2-2,pm-k^2-1)$, which are not integers. This gives a contradiction.

%

\bigskip\noindent
Case \romannumeral2: $k$ is odd.

If $n=pm-k^2$, Formula (\ref{dinv1}) implies
\[d(L(pm-k^2,1),0)=d(M,\mathfrak{t}_M)-2V_{\xi_0}.\]
We compute
\[V_{\xi_0}=\frac{k^2-1-(p-1)m}{8} \leq \frac{k^2-1-(p-1)(k+3)}{8}=\frac{k(k+1)+2-pk-3p}{8}<0,\]
which is a contradiction.

If $n=-pm+k^2$, Formula (\ref{dinv1}) gives
\[-d(L(pm-k^2,1),0)=d(M,\mathfrak{t}_M)-2V_{\xi_0},\]
which implies
\[V_{\xi_0}=\frac{(p+1)m-k^2-3}{8} \geq \frac{(p+1)(k+3)-k^2-3}{8}= \frac{pk+3p+k-k^2}{8} \geq 2.\]
We can thus apply Formula (\ref{dinv2}) and get
\[-d(L(pm-k^2,1), 0+i^{\ast}PD[\mu])=d(M,\mathfrak{t}_M+i^{\ast}PD[\mu])-2V_{\xi_0+PD[\mu]}.\]
Here, $-d(L(pm-k^2,1),0+i^{\ast}PD[\mu])= \frac{1}{4}-\frac{(2j-(pm-k^2))^2}{4(pm-k^2)}$ for some $j \in [0,pm-k^2-1]$.
Therefore
\begin{equation}
\label{eq20}
\frac{1}{4}-\frac{(2j-(pm-k^2))^2}{4(pm-k^2)}=\frac{pm^2-(6p+k^2)m+4p+6k^2}{4(mp-k^2)}-2V_{\xi_0+PD[\mu]}
\end{equation}
Since $V_{\xi_0}-1 \leq V_{\xi_0+PD[\mu]} \leq V_{\xi_0}$, we further divide it into two cases:

Case \romannumeral2(a): $V_{\xi_0+PD[\mu]}=V_{\xi_0}=\frac{(p+1)m-k^2-3}{8}$. Equation (\ref{eq20}) can be simplified to
\[j^2-(pm-k^2)j+p+k^2-mp=0\]
which is the same as Equation (\ref{eq26}). So the function $f(x)=x^2-(pm-k^2)x+p+k^2-mp$ has no root in $[0,pm-k^2-1]$, and this case can be ruled out.

\medskip
Case \romannumeral2(b): $V_{\xi_0+PD[\mu]}=V_{\xi_0}-1=\frac{(p+1)m-k^2-3}{8}-1$. Equation (\ref{eq20}) can be simplified to
\[j^2-(pm-k^2)j+p+pm-k^2=0,\]
which is the same as Equation (\ref{eq27}). Therefore the same argument can be applied to rule out this case.  This finishes all the cases and the proof.
\end{proof}

\section{Distance one surgeries on $L(5,1)$ and $L(7,1)$}
\label{Distance one surgeries on $L(5,1)$ and $L(7,1)$}
DNA knots and links in vivo experiments involving plasmids are often of small crossing numbers. This motivates our study of band surgeries from $T(2,5)$ and $T(2,7)$ to $T(2,n)$ in this section.

\subsection{Distance one surgeries on $L(5,1)$}
\label{L(5,1)}
As Theorem \ref{theoremgen} (\romannumeral1)(\romannumeral2) has handled the case $k=0,1$, we only need to consider the case $k=2$.  By (\ref{firsthomology}), the first homology of $L(n,1)$ is $|n|=|5m-4|$. As Theorem \ref{theoremeven} has completely solved the even $n$ case, we are left with the odd $n$ case.  So we assume that $m$ is an odd integer.  

If $m \leq -1$, we use the negative-definite plumbing diagram in Figure \ref{case1plumbing} to compute the $d$-invariant of the Seifert fiber space $M=M(0,0;(m-2,1),(3,1),(2,1))$.  We also compute its intersection form $Q_M$.

\begin{figure}[h!]
\begin{minipage}{0.48\linewidth}
 \centerline{\includegraphics[width=4.7cm]{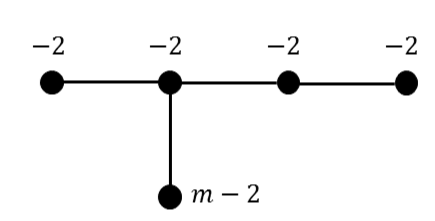}}
\end{minipage}
\hfill
\begin{minipage}{.48\linewidth}
$Q_M=
\begin{bmatrix}
 -2 & 1 &    &     &   \\
 1  & -2&    &     & 1 \\
    &   & -2 &     & 1 \\
    &   &    & m-2 & 1 \\
    & 1 &  1 &   1 & -2
\end{bmatrix}$
\end{minipage}
\caption{The negative-definite plumbing diagram and the corresponding intersection form of the Seifert fiber space $M(0,0;(m-2,1),(3,1),(2,1))$ for $m\leq -1$.}
\label{case1plumbing}
\end{figure}

We see that there is only one bad vertex in the graph, and $Q_M$ is negative definite. Applying Ozsv\'{a}th and Szab\'{o}'s algorithm, we can obtain $-5m+4$ candidates of maximiser by a similar argument as in Lemma \ref{maximiser}.

\begin{itemize}
\item $(2, 0, 0, j, 0)$, $(0, 0, 2, j, 0)$ and $(0,0,0,j,0)$ for $m \leq j \leq -m$.
\item $(0, 2, 0, j, 0)$ and $(0, 0, 0, j, 2)$ for $m \leq j \leq -m-2$.
\item $(0,0,0,-m+2,0)$.
\end{itemize}

Meanwhile, the order of the first homology of $(m \mu +\lambda)$-surgered manifold is precisely $-5m+4$. Therefore, each of these $-5m+4$ vectors must be the maximiser of Formula (\ref{dSFS}) for the corresponding $\rm Spin^c$ structure. By a similar argument of Lemma \ref{correpondingcovoter}, we can also show that the vector $(2,0,0,-m,0)$ (resp. $(0,0,2,m,0)$) corresponds to the $\rm Spin^c$ structure $\mathfrak{t}_M$ (resp. $\overline{\mathfrak{t}_M+i^{\ast}PD[\mu]}$).  We compute the $d$-invariant of $M$ as follows.

\[d(M,\mathfrak{t}_M)=\frac{m+1}{4}, \quad  d(M,\mathfrak{t}_M+i^{\ast}PD[\mu])=\frac{-5m^2-21m}{4(-5m+4)}.\]



\bigskip

\begin{proof}[Proof of Theorem \ref{theorem57} (\romannumeral1)]
Theorem \ref{theoremeven} implies that there is a distance one surgery from $L(5,1)$ to $L(n,1)$ with $n$ even if and only if $n=4,6$. When $n$ is odd and $K$ is null-homologous or $K$ is homologically essential with $k=1$, Theorem \ref{theoremgen} (\romannumeral1)(\romannumeral2) implies that $n$ can only be $\pm 1, \pm 5,-9$.  From now on, we assume that $n$ is odd and the winding number $k=2$. We divide the proof according to different values of $m$.

\medskip\noindent
(1) Suppose $m < -1$.  If $n=-5m+4$, Formula (\ref{dinv3}) gives
\[d(L(-5m+4,1),0)-d(M,\mathfrak{t}_M)=2V_{\xi_0},\]
which implies $V_{\xi_0}=\frac{-3m+1}{4} \geq 2$.  We can thus apply Formula (\ref{dinv4}) and get
\begin{equation}
\label{eq21}
d(L(-5m+4,1),0+i^{\ast}PD[\mu])-d(M,\mathfrak{t}_M+i^{\ast}PD[\mu])=2V_{\xi_0+PD[\mu]},
\end{equation}
where $d(L(-5m+4,1),0+i^{\ast}PD[\mu])=-\frac{1}{4}+\frac{(2j-(-5m+4))^2}{4(-5m+4)}$ for some $j \in [0,-5m+3]$.

\medskip
\noindent
Case \romannumeral1: $V_{\xi_0+PD[\mu]}=V_{\xi_0}=\frac{-3m+1}{4}$. Equation (\ref{eq21}) can be simplified to
\[j^2-(-5m+4)j+5m+1=0.\]
We claim that the function $f(x)=x^2-(-5m+4)x+5m+1$ has no root in $[0,-5m+3]$. Indeed, $f(x)<0$ for any $x \in [0,-5m+3]$ since its axis of symmetry is $x=\frac{-5m+4}{2}$ and $f(0)=5m+1<0$.

\medskip\noindent
Case \romannumeral2: $V_{\xi_0+PD[\mu]}=V_{\xi_0}-1=\frac{-3m+1}{4}-1$. Equation (\ref{eq21}) can be simplified to
\[j^2-(-5m+4)j-5m+9=0.\]
Let $f(x)=x^2-(-5m+4)x-5m+9$. The axis of symmetry of $f(x)$ is $x=\frac{-5m+4}{2}$, and $f(0)=-5m+9>0$, $f(1)=6>0$ and $f(2)=5m+5<0$ when $m<-1$, thus the roots of $f(x)$ lie in $(1,2)$ and $(-5m+2,-5m+3)$, which are not integers.


If $n=5m-4$, Formula (\ref{dinv3}) gives
\[-d(L(-5m+4,1),0)-d(M,\mathfrak{t}_M)=2V_{\xi_0},\]
which implies $V_{\xi_0}=\frac{m-1}{2} <0$. This gives a contradiction.  Hence, we just proved that when $m<-1$ there is no desired distance one surgery.

\medskip\noindent
(2) When $m=-1$, we have $|n|=9$.  There is a distance one surgery from $L(5,1)$ to $L(9,1)$ given by the double branched cover of the band surgery between $T(2,5)$ and $T(2,9)$ as Figure \ref{bandsurgery5}. If $n=-9$, Formula (\ref{dinv3}) gives
\[-d(L(9,1),0)-d(M,\mathfrak{t}_M)=2V_{\xi_0},\]
which implies $V_{\xi_0}=-1<0$. This gives a contradiction.

\medskip\noindent
(3) When $m=1$, we have $|n|=1$.  There is a distance one surgery from $L(5,1)$ to $S^3$ given by the double branched cover of the band surgery shown in Figure \ref{bandsurgery5} (with $p=1$ and the reverse direction).

\medskip\noindent
(4) When $m=3$, we have $|n|=11$. The Seifert fibered manifold $M$ is actually a lens space: $M(0,0;(1,1),(3,1),(2,1))=L(11,3)$.  If $n=11$, then by Formula (\ref{dinv1}) we have
\[d(L(11,1),0)=d(L(11,3),1)-2V_{\xi_0}.\]
Thus $V_{\xi_0}=-1$ since $d(L(11,1),0)=\frac{5}{2}$ and $d(L(11,3),1)=\frac{1}{2}$, which contradicts $V_{\xi_0} \geq 0$.
If $n=-11$, Formula (\ref{dinv1}) gives
\[-d(L(11,1),0)=d(L(11,3),1)-2V_{\xi_0}.\]
Thus $V_{\xi_0}=\frac{3}{2}$, which is impossible since $V_{\xi_0}$ is supposed to be an integer.

\medskip\noindent
(5) When $m\geq 5$, Theorem \ref{theoremgen} (\romannumeral3) implies that no solution exists in this case.

\medskip
In summary, we conclude that the lens space $L(n,1)$ is obtained by a distance one surgery from $L(5,1)$ only if $n=\pm 1, 4, \pm 5, 6, \pm 9$; except for $n=-9$, distance one surgeries to all other $L(n,1)$ on this list can be realized as the double branched covers of the band surgeries in Figure \ref{bandsurgery}. Unfortunately, our Heegaard Floer $d$-invariant obstruction fails for the case of $L(5,1)$ to $L(-9,1)$ when the surgery is performed on a homologically essential knot $K$ with $k=1$.
\end{proof}




\subsection{Distance one surgeries on $L(7,1)$}
\label{L(7,1)}
As Theorem \ref{theoremgen} (\romannumeral1)(\romannumeral2) has handled the case $k=0,1$, we only need to consider the case $k=2$ or $3$.  When $m<1$, we use the negative-definite plumbing diagram in Figure \ref{k=2plumbing} and \ref{k=3plumbing} to compute the $d$-invariant of the Seifert fiber spaces $M=M(0,0;(m-2,1),(5,1),(2,1))$ and $M'=M(0,0; (m-3,1), (4,1), (3,1))$, corresponding to $k=2$ and $k=3$, respectively. We also compute their respective intersection forms.



\begin{figure}[!h]
\begin{minipage}{0.48\linewidth}
 \centerline{\includegraphics[width=5cm]{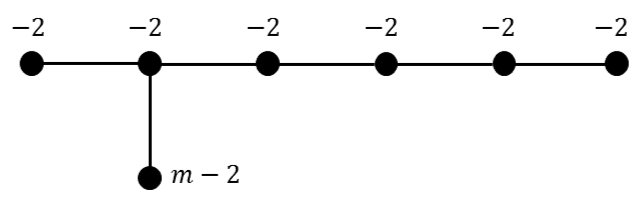}}
\end{minipage}
\hfill
\begin{minipage}{.48\linewidth}
$Q_M=
\begin{bmatrix}
 -2 & 1 &    &     &    &     &   \\
 1  & -2&  1 &     &    &     &   \\
    &  1& -2 &  1  &    &     &   \\
    &   &  1 &  -2 &    &     & 1 \\
    &   &    &     & -2 &     & 1 \\
    &   &    &     &    & m-2 & 1 \\
    &   &    &  1  &  1 &   1 &-2
\end{bmatrix}$
\end{minipage}
\caption{The plumbing diagrams and the corresponding intersection form of the Seifert fiber space $M=M(0,0;(m-2,1),(5,1),(2,1))$.}
\label{k=2plumbing}
\end{figure}

\begin{figure}[!h]
\begin{minipage}{0.48\linewidth}
 \centerline{\includegraphics[width=5cm]{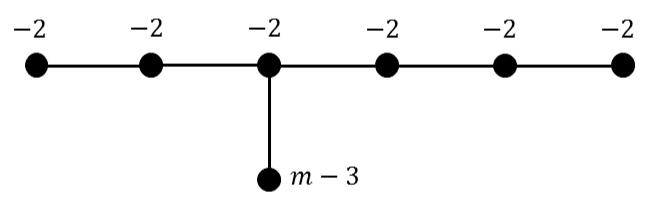}}
\end{minipage}
\hfill
\begin{minipage}{.48\linewidth}
$Q_{M'}=
\begin{bmatrix}
 -2 & 1 &    &     &    &     &   \\
 1  & -2&  1 &     &    &     &   \\
    &  1& -2 &     &    &     & 1 \\
    &   &    &  -2 &  1 &     &   \\
    &   &    &   1 & -2 &     & 1 \\
    &   &    &     &    & m-3 & 1 \\
    &   &  1 &     &  1 &   1 &-2
\end{bmatrix}$
\end{minipage}
\caption{The plumbing diagram and the corresponding intersection form of the Seifert fiber space $M'=M(0,0; (m-3,1), (4,1), (3,1))$.}
\label{k=3plumbing}
\end{figure}

Note that there is only one bad vertex in each of the two plumbing graphs, and both $Q_M$ and $Q_{M'}$ are negative definite. Similar to the case $L(5,1)$, we can obtain $-7m+4$ and $-7m+9$ maximisers for each case. Moreover, we can show that the vector $(2,0,0,0,0,-m,0)$ (resp. $(0,0,0,0,2,-m-1,0)$) corresponds to the $\rm Spin^c$ structure $\mathfrak{t}_M$ (resp. $\mathfrak{t}_{M'}$), and the vector $(0,0,0,0,2,m,0)$ (resp. $(0,2,0,0,0,m-1,0)$)
corresponds to the $\rm Spin^c$ structure $\overline{\mathfrak{t}_M+i^{\ast}PD[\mu]}$ (resp. $\overline{\mathfrak{t}_{M'}+i^{\ast}PD[\mu]}$) up to $\rm Spin^c$-conjugation.  Then we compute the $d$-invariant of $M$ and $M'$ as follows.

\[d(M,\mathfrak{t}_M)=\frac{m+3}{4}, \quad  d(M,\mathfrak{t}_M+i^{\ast}PD[\mu])=\frac{-7m^2-45m}{4(-7m+4)}.\]

\[d(M',\mathfrak{t}_{M'})=\frac{m}{4}, \quad  d(M',\mathfrak{t}_{M'}+i^{\ast}PD[\mu])=\frac{-7m^2-19m+8}{4(-7m+9)}.\]

\bigskip



Before we prove Theorem \ref{theorem57} (\romannumeral2), let us recall some facts of the linking form $lk: H_1(Y) \times H_1(Y) \rightarrow \mathbb{Q}/\mathbb{Z}$ for a rational homology sphere $Y$, which we also use as an obstruction in this case. If a rational homology sphere $Y$ has cyclic first homology, we can use a fraction to represent its linking form, which is the value $lk(x,x)$ for a generator $x$ of the first homology group. Let $N=Y-K$ be the exterior of a primitive knot $K$ in $Y$. Then $H_1(N)=\mathbb{Z}$. Choose two curves $m$ and $l$ on $\partial N$ such that $(m,l)$ is a basis of $H_1(\partial N)$, where $l$ is null-homologous in $N$ and $m$ generates $H_1(N)$. Let $N(pm+ql)$ denote the Dehn filling of $N$ along the curve $pm+ql$. The linking form of $N(pm+ql)$ is $\frac{q}{p}$ if $p \neq 0$. If two rational homology spheres $Y_1$ and $Y_2$ have cyclic first homology group with linking forms $ \frac{q_1}{p}$ and $\frac{q_2}{p}$ for $p>0$, then the two forms are equivalent if and only if $q_1 \equiv q_2 a^2 \,\, (mod \,\,p)$ for some integer $a$ with $gcd(a,p)=1$.

\bigskip

\begin{proof}[Proof of Theorem \ref{theorem57} (\romannumeral2)]
Theorem \ref{theoremeven} implies that there is a distance one surgery from $L(7,1)$ to $L(n,1)$ with $n$ even if and only if $n=6,8$. Theorem \ref{theoremgen} (\romannumeral1)(\romannumeral2) and the band surgeries we construct in Figure \ref{bandsurgery} shows that when $n$ is odd and $K$ is null-homologous or homologically essential with $k=1$, there exists a distance one surgery if and only if $n=\pm 1, 7$. From now on, we assume that $n$ is odd and the winding number $k=2$ or $3$.  By (\ref{firsthomology}), the order of the first homology of $L(n,1)$ equals $|n|=|7m-k^2|$.  Hence, when $k=2$, $m$ is odd; when $k=3$, $m$ is even.  Subsequently, we divide the proof according to different values of $k$ and $m$.

\medskip \noindent
(1) Suppose $k=2$ and $m < -1$.  If $n=-7m+4$, Formula (\ref{dinv3}) gives
\[d(L(-7m+4,1),0)-d(M,\mathfrak{t}_M)=2V_{\xi_0},\]
which implies $V_{\xi_0}=-m \geq 2$.  We can thus apply Formula (\ref{dinv4}) and get
\begin{equation}
\label{eq22}
d(L(-7m+4,1),0+i^{\ast}PD[\mu])-d(M,\mathfrak{t}_M+i^{\ast}PD[\mu])=2V_{\xi_0+PD[\mu]},
\end{equation}
where $d(L(-7m+4,1),0+i^{\ast}PD[\mu])=-\frac{1}{4}+\frac{(2j-(-7m+4))^2}{4(-7m+4)}$ for some $j \in [0,-7m+3]$.

\medskip
\noindent
Case \romannumeral1: $V_{\xi_0+PD[\mu]}=V_{\xi_0}=-m$. Equation (\ref{eq22}) can be simplified to
\[j^2-(-7m+4)j+7m+3=0.\]
However, the function $f(x)=x^2-(-7m+4)x+7m+3$ has no root in $[0,-7m+3]$. Indeed, $f(x)<0$ for any $x \in [0,-7m+3]$ since its axis of symmetry is $x=\frac{-7m+4}{2}$ and $f(0)=7m+3<0$. This give a contradiction.

\medskip
\noindent
Case \romannumeral2: $V_{\xi_0+PD[\mu]}=V_{\xi_0}-1=-m-1$. Equation (\ref{eq22}) can be simplified to
\[j^2-(-7m+4)j-7m+11=0.\]
Let $f(x)=x^2-(-7m+4)x-7m+11$. The axis of symmetry of $f(x)$ is $x=\frac{-7m+4}{2}$, and $f(0)=-7m+11>0$, $f(1)=8>0$ and $f(2)=7m+7<0$. Therefore the roots of $f(x)$ lie in $(1,2)$ and $(-7m+2, -7m+3)$, which are not integers.


If $n=7m-4$, Formula (\ref{dinv3}) gives
\[-d(L(-7m+4,1),0)-d(M,\mathfrak{t}_M)=2V_{\xi_0},\]
which implies $V_{\xi_0}=\frac{3m-3}{4}<0$. This gives a contradiction.

\medskip
\noindent
(2) Suppose $k=2$ and $m=-1$. Then $|n|=11$. There is a distance one surgery from $L(7,1)$ to $L(11,1)$ given by the double branched cover of the band surgery between $T(2,7)$ and $T(2,11)$ in Figure \ref{bandsurgery5}. If $n=-11$, Formula (\ref{dinv3}) gives
\[-d(L(11,1),0)-d(M,\mathfrak{t}_M)=2V_{\xi_0},\]
which implies $V_{\xi_0}=-\frac{3}{2}$. This gives a contradiction.


\medskip
\noindent
(3) Suppose $k=2$ and $m=1$.  By \cite{LMV}, there is no distance one surgery from $L(7,1)$ to $L(-3,1)$.  On the other hand, a distance one surgery along the simple knot $K(7,1,2)$ produces $L(3,1)$.

\medskip
\noindent
(4) Suppose $k=2$ and $m=3$.  Then $|n|=|7m-k^2|=17$, so the desired lens space is $L(17,1)$ or $L(-17,1)$. If $n=17$, then by Formula (\ref{dinv1}) we have
\[d(L(17,1),0)=d(L(17,3),1)-2V_{\xi_0},\]
which implies $V_{\xi_0}=-\frac{3}{2}<0$.  If $n=-17$, then Formula (\ref{dinv1}) gives \[-d(L(17,1),0)=d(L(17,3),1)-2V_{\xi_0}.\]
Thus $V_{\xi_0}=\frac{5}{2}$, which is impossible since $V_{\xi_0}$ is supposed to be an integer.

\medskip
\noindent
(5) Suppose $k=2$ and $m\geq 5$.  Theorem \ref{theoremgen} (\romannumeral3) implies that there is no solution in this case.

\medskip
\noindent
(6) Suppose $k=3$ and $m < 0$.  If $n=-7m+9$, Formula (\ref{dinv3}) gives
\[d(L(-7m+9,1),0)-d(M',\mathfrak{t}_{M'})=2V_{\xi_0},\]
which implies $V_{\xi_0}=-m+1\geq 2$. We can thus apply Formula (\ref{dinv4}) and get
\begin{equation}
\label{eq23}
d(L(-7m+9,1),0+i^{\ast}PD[\mu])-d(M',\mathfrak{t}_{M'}+i^{\ast}PD[\mu])=2V_{\xi_0+PD[\mu]},
\end{equation}
where $d(L(-7m+9,1),0+i^{\ast}PD[\mu])=-\frac{1}{4}+\frac{(2j-(-7m+9))^2}{4(-7m+9)}$ for some $j \in [0,-7m+8]$.

\medskip
\noindent
Case \romannumeral1: $V_{\xi_0+PD[\mu]}=V_{\xi_0}=-m+1$.  Equation (\ref{eq23}) can be simplified to
\[j^2-(-7m+9)j+7m-2=0.\]
However, the function $f(x)=x^2-(-7m+9)x+7m-2$ has no root in $[0,-7m+8]$. Indeed, $f(x)<0$ for any $x \in [-7m+8]$ since its axis of symmetry is $x=\frac{-7m+9}{2}$ and $f(0)=7m-2<0$. This give a contradiction.

\medskip
\noindent
Case \romannumeral2: $V_{\xi_0+PD[\mu]}=V_{\xi_0}-1=-m$.  Equation (\ref{eq23}) can be simplified to
\[j^2-(-7m+9)j-7m+16=0.\]
Let $f(x)=x^2-(-7m+9)x-7m+16$. The axis of symmetry of $f(x)$ is $x=\frac{-7m+9}{2}$, and $f(0)=-7m+16>0$, $f(1)=8>0$ and $f(2)=7m+2<0$. Therefore the roots of $f(x)$ lie in $(1,2)$ and $(-7m+7,-7m+8)$, which are not integers.

\medskip

If $n=7m-9$, Formula (\ref{dinv3}) gives
\begin{equation*}
-d(L(-7m+9,1),0)-d(M',\mathfrak{t}_{M'})=2V_{\xi_0}, \label{eq30}
\end{equation*}
which implies $V_{\xi_0}=\frac{3m-4}{4}<0$. This gives a contradiction.


\medskip
\noindent
(7) Suppose $k=3$ and $m=0$ (i.e. doing $\lambda$-surgery). If $n=7m-9=-9$, Formula (\ref{dinv3}) gives
\begin{equation*}
-d(L(9,1),0)-d(M',\mathfrak{t}_{M'})=2V_{\xi_0}, \label{eq30}
\end{equation*}
which implies $V_{\xi_0}=-1<0$. This gives a contradiction. If $n=-7m+9=9$, we use the linking form to obstruct this case. Let $m=4\mu+3\lambda$ and $l=9\mu+7\lambda$. One may check that $(m,l)$ forms a basis of $H_1(\partial (L(7,1)-K))$, where $m$ generates $H_1(L(7,1)-K)$ and $l$ is null-homologous in $L(7,1)-K$. We see $\lambda=-9m+4l$, thus the linking form of the surgered manifold is $\frac{-4}{9}$. However the linking form of the desired lens space $L(9,1)$ is $\frac{1}{9}$, and $-4$ is not a quadratic residue modulo $9$.


\medskip
\noindent
(8) Suppose $k=3$ and $m=2$. By our discussion in the previous section for distance one surgery on $L(5,1)$, we see that there is no distance one surgery from $L(7,1)$ to $L(\pm 5,1)$.

\medskip
\noindent
(9) Suppose $k=3$ and $m=4$.  Then $|n|=|7m-k^2|=19$, so the desired lens space is $L(19,1)$ or $L(-19,1)$. If $n=19$, then by Formula (\ref{dinv1}) we have
\[d(L(19,1),0)=d(L(19,4),11)-2V_{\xi_0},\]
which implies $V_{\xi_0}=-2<0$.  If $n=-19$, then Formula (\ref{dinv1}) gives
\[-d(L(19,1),0)=d(L(19,4),11)-2V_{\xi_0}.\]
Thus $V_{\xi_0}=\frac{5}{2}$, which is impossible since $V_{\xi_0}$ is supposed to be an integer.

\medskip
\noindent
(10) Suppose $k=3$ and $m\geq 6$.  Theorem \ref{theoremgen} (\romannumeral3) implies that there is no solution in this case.

\medskip
In summary, we conclude that the lens space $L(n,1)$ is obtained by a distance one surgery from $L(7,1)$ if and only if $n=\pm 1, 3, 6, 7, 8$ or $11$.
\end{proof}







\begin{thebibliography}{10}

\bibitem{Boil} M. Boileau, S. Boyer, R. Cebanu and G. S. Walsh,  \textit{Knot commensurability and the Berge conjecture.} Geom. Topol., 16(2): 625-664, 2012.
\bibitem{Darcysumner} I. K. Darcy, D. W. Sumners \textit{Rational tangle distances on knots and links.} Math. Proc. Cambridge Philos. Soc., 128(3): 497-510, 2000.
\bibitem{Greene} J. E. Greene, \textit{The lens space realization problem.} Ann. of Math. (2), 177(2):449-511, 2013.
\bibitem {Lick} W. R. Lickorish, \textit{A representation of orientable combinatorial 3-manifolds.} Ann. of Math. (2), 76: 531-540, 1962.
\bibitem{LMV} T. Lidman, A. H. Moore and M. Vazquez,  \textit{Distance one lens space fillings and band surgery on the trefoil knot.} Alg. Geom. Topol., 19(5): 2439-2484, 2019.
\bibitem{MV} A. H. Moore and M. Vazquez,  \textit{A note on band surgery and the signature of a knot.} arXiv:1806.02440, 2018.
\bibitem{NiWu} Y. Ni and Z. Wu,   \textit{Cosmetic surgeries on knots in $S^3$.} J. Reine Angew. Math., 706: 1-17, 2015.
\bibitem{OS1} P. Ozsv\'{a}th and Z. Szab\'{o}, \textit{Absolutely graded Floer homologies and intersection forms for four-manifolds with boundary.} Adv. Math., 173(2): 179-261, 2003.
\bibitem{OSd} P. Ozsv\'{a}th and Z. Szab\'{o}, \textit{On the Floer homology of plumbed three-manifolds.} Geom. Topol., 7(1): 185-224, 2003.
\bibitem{OSi} P. Ozsv\'{a}th and Z. Szab\'{o},  \textit{Knot Floer homology and integer surgeries.} Alg. Geom. Topol., 8(1): 101-153, 2008.
\bibitem{OSr} P. Ozsv\'{a}th and Z. Szab\'{o},  \textit{Knot Floer homology and rational surgeries.} Alg. Geom. Topol., 11(1): 1-68, 2010.
\bibitem{Ras1} J. Rasmussen, \textit{Floer homology and knot complements.} Ph.D. Thesis, arXiv:math/0306378, 2003.
\bibitem{Ras2} J. Rasmussen,  \textit{Lens space surgeries and L-space homology spheres.} arXiv:0710.2531, 2007.
\bibitem{Tur} V. Turaev, \textit{Torsions of 3-dimensional manifolds.} Progress in Math. 208, $\rm Birkh\ddot{a}user$ Verlag, Basel, MR1958479, 2002.
\bibitem {Wal} A. H. Wallace, \textit{Modifications and cobounding manifolds.} Canad. J. Math. 12: 503-528, 1960.
\end{thebibliography}
\end{document}